\documentclass[10pt,a4paper,reqno]{amsart}

\usepackage{dsfont}
\usepackage{amssymb,graphicx,xcolor}
\usepackage[colorlinks=true, linkcolor=magenta, citecolor=cyan, urlcolor=blue]{hyperref}
\usepackage{amssymb,bbm,enumerate}
\usepackage{esint}
\usepackage{colortbl}
\usepackage{xcolor}
\usepackage[utf8]{inputenc}

\usepackage{amsmath,amssymb,amstext,amsthm}

\numberwithin{equation}{section}

\newtheorem{theorem}{Theorem}[section]

\newtheorem{lemma}[theorem]{Lemma}

\newtheorem{proposition}[theorem]{Proposition}
\theoremstyle{remark}
\newtheorem{remark}[theorem]{Remark}

\theoremstyle{definition}
\newtheorem{definition}[theorem]{Definition}
\newtheorem{assumption}[theorem]{Assumption}

\newtheorem{claim}{Claim}

\newcommand{\R}{\mathbb{R}}

\DeclareMathOperator{\diam}{diam}

\title[Polynomial mixing under a certain stationary Euler flow]{Polynomial mixing\\ under a certain stationary Euler flow}
\date{\today}    

\author{Gianluca Crippa}
\address{Gianluca Crippa\\
	Department Mathematik und Informatik, Universität Basel}
\curraddr{Spiegelgasse 1, CH-4051 Basel, Switzerland}
\email{gianluca.crippa@unibas.ch}

\author{Renato Luc\`a}
\address{Renato Luc\`a\\
	Department Mathematik und Informatik, Universität Basel}
\curraddr{Spiegelgasse 1, CH-4051 Basel, Switzerland}
\email{renato.luca@unibas.ch}

\author{Christian Schulze}
\address{Christian Schulze\\
	Department Mathematik und Informatik, Universität Basel}
\curraddr{Spiegelgasse 1, CH-4051 Basel, Switzerland}
\email{christian.schulze@unibas.ch}


\begin{document}

\begin{abstract}
We study the mixing properties of a scalar $\rho$ on the unit disk advected by a certain incompressible velocity field $u$, which is a stationary radial solution of the Euler equation. The scalar $\rho$ solves the continuity equation with the velocity field $u$ and we can measure the degree 
of \lq\lq mixedness'' of~$\rho$ with two different scales commonly used in this setting, namely the geometric and the functional mixing scale. We develop a physical space approach well adapted to the quantitative analysis of the decay in time of the geometric mixing scale, which turns out to be polynomial for a large class of initial data. This extends previous results for the functional mixing scale, based on the explicit expression for the solution in Fourier variable, results that are also partially recovered by our approach.
%
%
%
\end{abstract}

\maketitle

\section{Introduction}
We consider a passive scalar $\rho$ (also called tracer) on the two--dimensional unit disk $B_{1}(0)$, advected by a  divergence-free velocity field $u$ which is tangent to the boundary $\partial B_1(0)$. 
Given a mean-free initial condition $\rho_0$, the scalar $\rho$ satisfies the Cauchy problem for the continuity equation with velocity field $u$:
\begin{equation}
\label{Cauchyprob}
\begin{cases} \partial_t\rho + \textnormal{div} (u\rho) =0 & \mbox{on } [0,\infty) \times B_{1}(0) \\ \rho(0,\cdot)=\rho_{0} & \mbox{on }B_{1}(0) .\end{cases}
\end{equation}	
Observe that the mean-free condition for the tracer is preserved by the time evolution. 

In this note we study certain mixing properties of the solution $\rho$ under the action of the following autonomous velocity field  
\begin{equation}
\label{themixa}
u(t,r,\theta) = (u_{1}(r,\theta), u_{2}(r,\theta)) := 2 \pi r^{2} ( \sin \theta, - \cos \theta), \qquad t \geq 0 \, ,
\end{equation}
where $(r, \theta)$ are polar coordinates. Notice that $u$ is a smooth stationary solution to the two dimensional Euler equation
\begin{equation}
\partial_{t} u + ( u \cdot \nabla ) u = - \nabla P , \qquad \textnormal{div } u = 0 \, ,  
\end{equation}
with pressure $P = -|u|^{2}/2 + \mbox{const}$.

In fact, this velocity field is the canonical counterpart on the unit disk of a shear flow on the two dimensional flat torus. Mixing by shear flows has been studied in a variety of settings and geometries, most recently in connection with inviscid damping for the Euler equation (see in particular \cite{Bedrossian, Zill1, Zill2, Zeng} and the references therein). Heuristically, for the velocity field in \eqref{themixa}, mixing is due to the fact that, as a consequence of the increase of the angular component of $u$ in the radial direction, different portions of the tracer move close to others with different history and thus relatively different concentrations; see Figure~\ref{f:example}. 

	\begin{figure}\label{GreatFigure}
		\centering
		\begin{tabular}[b]{c}
			\includegraphics[width=.25\linewidth]{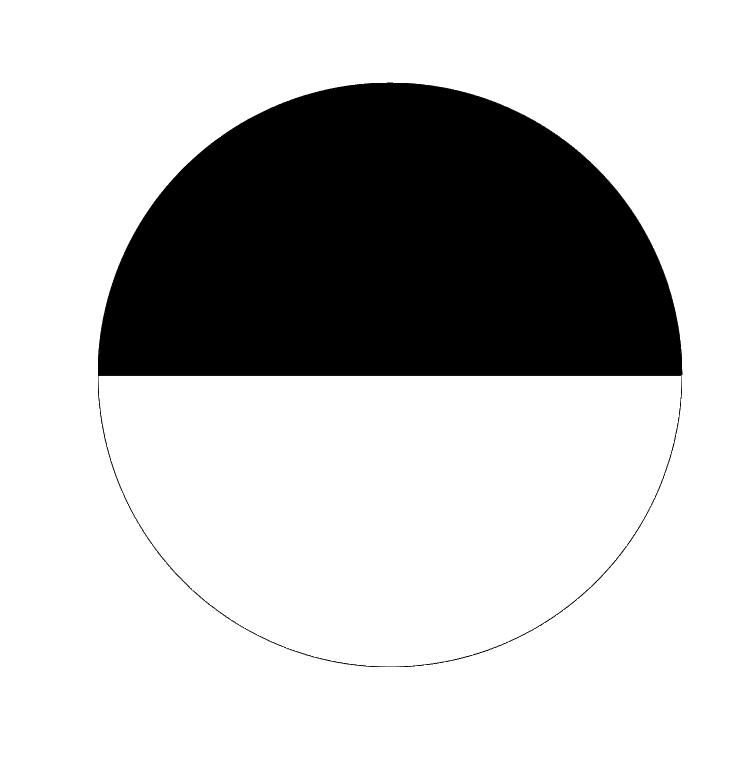} \\
			\small $t=0$
		\end{tabular} \qquad
		\hspace{1cm}
		\begin{tabular}[b]{c}
			\includegraphics[width=.25\linewidth]{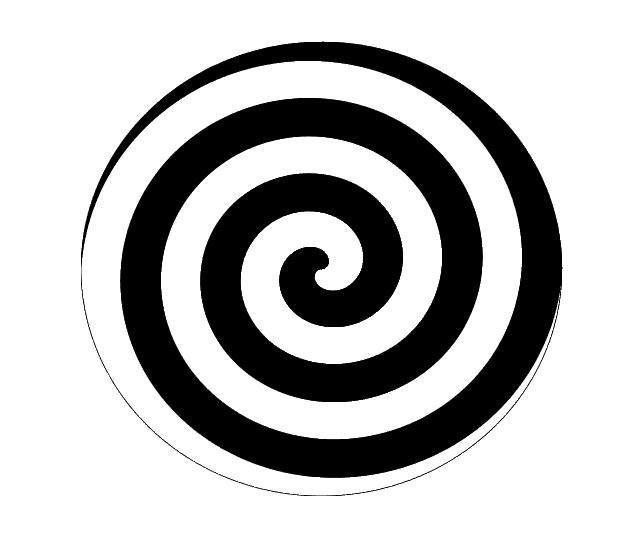} \\
			\small $t=5$
		\end{tabular}
		\caption{An example of evolution under the action of the velocity field \eqref{themixa}. We have $\rho = 1$ in the black region and
		$\rho = -1$ in the white region. \label{f:example}}
	\end{figure}

In this note we measure the degree of \lq\lq mixedness'' of the tracer $\rho$ with two different mixing scales which are commonly used in this setting. The first one is the \emph{geometric} mixing scale $\mathcal{G}(\cdot)$ introduced in \cite{Bressan}: 
\begin{definition}[Geometric Mixing Scale]
	\label{GeomS}
	Given an accuracy parameter $0<\kappa<1$, the geometric mixing scale of $\rho(t,\cdot)$ is the infimum $\varepsilon(t)$ of all $\varepsilon>0$ such that for every $x\in \mathbb{R}^2$ there holds 
	\begin{equation}\label{GeoScale}
	\left|\,\fint_{B_\varepsilon(x)}\rho(t,y)\,dy\right|\leq \kappa \|\rho(t,\cdot)\|_{L^\infty(B_1(0))} \, .
	\end{equation}
	We denote 
	$$ \mathcal{G}(\rho(t,\cdot)):= \varepsilon(t) \, . $$
\end{definition}

We systematically use the notation $B_r(x) \subset \mathbb{R}^{2}$ for the two--dimensional open disk centered at $x$ with radius~$r$ and we abbreviate $B_r(0)$ to $B_{r}$. 
The parameter $\kappa$, which measures the accuracy of the mixing, can be exploited in order to avoid pathological examples like the one 
discussed in Remark \ref{Pato}. 

The second mixing scale we use is the \emph{functional} one, which has been introduced in \cite{Multi} and subsequently widely employed in the applied fluid dynamics literature (see for instance the survey \cite{Thieff}): 
\begin{definition}[Functional Mixing Scale]
	\label{FuMS}
	The functional mixing scale of~$\rho(t,\cdot)$ is~$\|\rho(t,\cdot)\|_{\dot{H}^{-1}(B_1)}$.
\end{definition}
In the above definition, $\|\cdot\|_{\dot{H}^{-1}(B_1)}$ denotes the semi-norm in the negative 
homogeneous Sobolev space, defined in \eqref{tinaf}. Since we will always restrict to zero-mean functions, this 
actually turns to be a norm for our purposes. 

We observe that, although intuitively related, the two mixing scales in Definitions~\ref{GeomS} and \ref{FuMS} are not equivalent; see the examples and remarks in \cite{Lin}. The link between these two notions has been extensively analyzed in \cite{Zill4}. 

\medskip

We are interested in quantifying the weak convergence to zero (i.e., the average of the initial datum) of the solution of~\eqref{Cauchyprob} with the velocity field $u$ in \eqref{themixa}, that is, in quantifying the decay to zero (as a function of time) of the two mixing scales in Definitions~\ref{GeomS} and \ref{FuMS} for the solution. 

\begin{remark}\label{r:atmost}
It can be seen quite easily that the decay of any of the two mixing scales under consideration cannot be faster than polynomial. Indeed, such decay is controlled by the regularity with respect to the space variable of the ODE flow associated to $u$, and it is immediate to check that the Lipschitz constant of the flow grows linearly in time. Heuristically this ``slow mixing'' is due to the fact that the velocity field, being time-independent, can stir the solution at each point in one direction only. Due to this structural constraint the decay rate is therefore much slower than the exponential rate typically associated to self-similar (and therefore, heavily time-dependent) evolutions, a brief account of which will be given in Remark~\ref{r:exp}.
\end{remark}

\medskip

Let us consider the following assumption on the initial datum $\rho_0$ in \eqref{Cauchyprob}:
\begin{assumption}\label{equalradius3}
We assume that $\rho_0$ is a bounded function which is zero outside $\overline{B_1}$ and which satisfies the following condition of zero average on circles: 
		\begin{equation}
		\int_{\partial B_r}\rho_0\,dS_{r}=0 \, 
		\end{equation}
		for almost every $r>0$, where $dS_{r}$ is the uniform measure on the circle of radius $r$.
\end{assumption}
Under Assumption~\ref{equalradius3}, the argument in \cite[Theorem 2.1]{Zill1} (see also \cite{Zill2}) based on the explicit formula in Fourier variable for the solution $\rho$ gives that for any initial datum $\rho_0 \in L^2(B_1)$ the functional mixing scale of the solution converges to zero, i.e., $\|\rho(t,\cdot)\|_{\dot{H}^{-1}(B_1)} \to 0$. Assuming some regularity on the initial datum $\rho_0$ the same argument gives a rate of convergence, more specifically 
\begin{equation}\label{e:optimal}
\|\rho(t,\cdot)\|_{\dot{H}^{-1}(B_1)} \leq C t^{-\alpha} \qquad \text{ for any } \rho_0 \in \dot{H}^{\alpha}(B_1) \,,
\end{equation}
and
\begin{equation}\label{e:exponent}
\|\rho(t,\cdot)\|_{\dot{H}^{-1}(B_1)} \leq C t^{-\alpha/2} \qquad \text{ for any } \rho_0 \in \dot{W}^{\alpha,1}(B_1) \,.
\end{equation}

\begin{remark} (i) Without Assumption~\ref{equalradius3} one can see that the solution converges weakly in $L^2(B_1)$ to the function taking on each circle the constant value equal to the average of $\rho_0$ on the circle itself. (ii) Polynomial decay of the functional mixing scale can be proved for more general velocity fields, under suitable nondegeneracy conditions on the profile of the velocity. This is technically more complicated and requires the use of the method of stationary phase for oscillatory integrals; see the Appendix of~\cite{Bedrossian}. (iii) By means of examples it is proved in~\cite{Zill4} the optimality (up to iterated logarithmic loss) of the rate in~\eqref{e:optimal}. 
\end{remark}

To the best of our understanding such Fourier variable techniques cannot be applied to analyze the decay of the {\em geometric} mixing scale of the solution. Our objective in this note is to develop an approach in physical space well adapted to the study of the geometric mixing scale. It essentially consists of two steps:
\begin{itemize}
\item[(1)] Explicit analysis of the mixing rate for some specific step functions, and
\item[(2)] Approximation of a general function with step functions as in (1).
\end{itemize}
In this procedure the accuracy $\kappa$ and the regularity of the data will play an important role. 
Indeed, they will both influence the scale at which we can perform the approximation procedure with step 
functions (see Seciton \ref{Sec:Proofs}) and the analysis 
of the mixing rate step functions at this given scale (see Proposition \ref{gridini}).

Besides allowing the analysis for the specific example considered in the present paper, we believe that our approach could be useful in broader settings, in which the presence of more general geometries and velocity profiles makes the use of Fourier analysis techniques unfeasible.

\medskip

The first result that we obtain with this approach is that every bounded initial datum satisfying Assumption~\ref{equalradius3} gets mixed by the velocity field we are considering:
\begin{theorem}[Universality of the mixer]
		\label{MT}
		For any initial datum $\rho_{0} \in L^{\infty}$ supported in $\overline{B_{1}} $ which satisfies
		Assumption~\ref{equalradius3} we have 
		\begin{equation}
		\label{Heri}
		\mathcal{G}(\rho(t,\cdot))\to 0 \quad \text{ and } \quad \|\rho(t,\cdot)\|_{\dot{H}^{-1}(B_1)}\to 0,
		\qquad \text{as} \quad t\to \infty \, .
		\end{equation}
\end{theorem}
 
We are not able to give a quantitative rate of decay for such a general class of initial data as in Theorem~\ref{MT}. However, in the case when the initial datum is continuous, or has fractional Sobolev regularity, the approximation step in (2) in our strategy can be made quantitative. This allows us to prove the following result:
	\begin{theorem}\label{ratethm}
		Let $\rho_0$ be as in Theorem \ref{MT}.  
	\begin{itemize}
		\item[(i)] If $\rho_0\in C(\overline{B_1})$ then there exists an absolute constant $C >0$ and a constant $\widetilde{C}$ 
		which depends on the datum $\rho$ and on the accuracy $\kappa$ such that
		\begin{equation}\label{ContRes} \mathcal{G}(\rho(t,\cdot)) \leq \frac{C}{\kappa^{2} } t^{-1}, \qquad  \mbox{for all  $t\geq \widetilde{C}$} \, .
		\end{equation}
		 \item[(ii)] If $\rho_0 \in \dot{W}^{\alpha,1}(B_1)$ with $\alpha\in (0,1]$, then 
		\begin{equation}\label{SobRes}
		\mathcal{G}(\rho(t,\cdot))\leq C_1 t^{-\frac{\alpha}{2}} , \qquad  \mbox{for all  $t\geq C_{2}$} \, ,
		\end{equation}
		where the constants here depend on $\alpha$, $\kappa$, $\|\rho_0\|_{\dot{W}^{\alpha,1}(B_1)}$ and $\|\rho_{0}\|_{L^{\infty}}$. 
		\end{itemize}
		\end{theorem}
Observe that \eqref{SobRes} entails the same rate as in \eqref{e:exponent} for the functional mixing scale. Moreover, let us stress that we obtain 
an explicit rate even for continuous functions, without requiring any fractional Sobolev regularity. In fact, recalling the discussion in Remark~\ref{r:atmost}, the decay rate in~\eqref{ContRes} turns out to be optimal. 

\medskip

In fact, it is possible to exploit our approach also for the analysis of the decay of the functional mixing scale. However, due to our method entailing an approximation step, we just obtain a decay rate slower than the one ensured by the exact computation in Fourier variable:		
		\begin{proposition}\label{rateprop}
		Let $\rho_0$ be as in Theorem \ref{MT}.  
			\begin{itemize}
			\item[(i)] If $\rho_0\in C^{0,\alpha}(B_1)$ with $\alpha\in (0,1]$, then  
		 	\begin{equation}\label{HoeldRes}
			\|\rho(t,\cdot)\|_{\dot{H}^{-1}(B_1)}\leq   C_{3}t^{-\frac{\alpha}{\alpha+1}}, 
			\qquad  \mbox{for all  $t\geq C_{4}$} \, ,
			\end{equation}
			where the constants here depend on $\alpha$, $\|\rho_0\|_{C^{0,\alpha}(B_1)}$ 
			and $\|\rho_{0}\|_{L^{\infty}}$.  			
                         \item[(ii)] If $\rho_0 \in \dot{W}^{\alpha,1}(B_1)$ with $\alpha\in (0,1]$, then 
			\begin{equation}\label{SobRes2}
			\|\rho(t,\cdot)\|_{\dot{H}^{-1}(B_1)}\leq  C_{5}t^{-\frac{\alpha}{\alpha+4}}, 
			\qquad  \mbox{for all  $t\geq C_{6}$}\, ,
			\end{equation}
			where the constants here depend on $\alpha$, $\|\rho_0\|_{\dot{W}^{\alpha,1}(B_1)}$ 
			and $\|\rho_{0}\|_{L^{\infty}}$.
			\end{itemize}
		\end{proposition}


\begin{remark}[condition of zero average on circles]
Without Assumption~\ref{equalradius3} the results of Theorems \ref{MT} and \ref{ratethm} and Proposition \ref{rateprop} cannot hold (for a fixed but arbitrary accuracy parameter~$\kappa$, in the case of the geometric mixing scale). Consider for instance an initial datum which is $-1$ on an 
inner disk and $+1$ on an outer annulus, as in Figure~\ref{Notmixing}. This particular example does not get mixed (indeed it is a stationary solution of  \eqref{Cauchyprob}). We prove in Proposition~\ref{NecCond} that Assumption~\ref{equalradius3} is in fact necessary in order for a bounded initial density to get mixed by the velocity field $u$. 
\end{remark}

\begin{figure}[h]
	\begin{center}
	\scalebox{0.25}{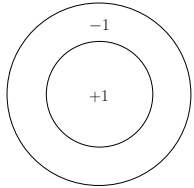}
\end{center}
		\caption{An example of an initial datum that is not mixed by the velocity field $u$. \label{Notmixing}}
	\end{figure}

\begin{remark}[role of the accuracy parameter $\kappa$]\label{Pato}
Given an accuracy parameter $\kappa$ we consider an initial datum which is equal to $\kappa$ on an inner
disk, $-\kappa$ on an intermediate annulus, and $+1$ and $-1$ on the upper and lower half of an outer annulus (see Figure \ref{ocho}). By a simple inspection of the proof of Proposition \ref{gridini} we see that the
geometric mixing scale $\mathcal{G}(\rho(t,\cdot))$ with accuracy parameter $\kappa$ decays like $1/t$.
However, the solution clearly does not converge to zero weakly in $L^2$ in the inner disk and in the intermediate annulus, where it is in fact stationary. To overcome this pathological behavior we notice that the geometric mixing scale does not go to zero as long as we choose any finer accuracy parameter $0< \kappa' < \kappa$. This suggests that also the accuracy $\kappa$ plays an important role in the analysis of mixing, see also \cite{Zill4}. Indeed, in Proposition \ref{WeakDecay} we show that if $\mathcal{G}(\rho(t,\cdot))$ decays to zero for 
all $\kappa \in (0,1)$, then the solution $\rho(t, \cdot)$ converges to zero weakly in $L^{2}$.
This is also equivalent to the decay to zero of the functional mixing scale, see \cite{Lin},
providing a quantitative measure of the level of mixedeness in the ergodic sense; see \cite{LinThif,Multi}.
\end{remark}

\begin{figure}[h]
	\begin{center}
		\scalebox{0.3}{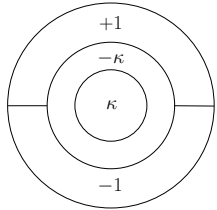}
	\end{center}
	\caption{Example of an initial datum whose geometric mixing scale goes to zero, but the functional mixing scale does not. \label{ocho}  }
\end{figure}

		\begin{remark}[behaviour of the constants]
We can not predict the behavior of the constant $\widetilde{C}$ that appears in~\eqref{ContRes}, which depends on the modulus
of continuity of the initial datum $\rho_{0}$. However, 
\eqref{ContRes} gives a precise asymptotic upper bound for the geometric mixing scale of continuous 
initial data, namely that
$$
\kappa^{2} \, t \, \limsup_{t\to \infty} \mathcal{G}(\rho(t,\cdot)) \leq C  \, , 
$$   
where $C$ is an absolute constant (in particular it is independent on $\rho_{0}$).
If we assume some fractional Sobolev regularity on $\rho_{0}$, we see that the geometric mixing scale decays at a polynomial rate that depends on the regularity of the initial data. 
It is worth to remark that, in contrast to $\widetilde{C}$,
all the constants $C_j$, $j=1, \ldots, 6$, will be explicitly estimated in the proofs of the inequalities \eqref{SobRes}, \eqref{HoeldRes}, and \eqref{SobRes2}.    
In particular, looking at \eqref{SobT}, \eqref{C2Diverge}, and \eqref{Sob2T}, we see that $C_{2}, C_{4}, C_{6} \to \infty$ as $\alpha \to 0$ and looking at \eqref{TchRhs}, \eqref{TchRhsBis}, \eqref{Finally1}, \eqref{TchRhsBisTris}, and \eqref{TchRhsBisFour}, we see that the 
constants $C_{1}, C_{3}, C_{5}$ are bounded as 
$\alpha \to 0$.  
		\end{remark}

%

\begin{remark}[exponential mixing under cellular velocity fields]\label{r:exp}
We have already commented on the fact that the rate of decay of the mixing scales is polynomial and not faster due to the strong constraint that the velocity field is smooth and time-independent. To put into context the results of this note we briefly review some of the 
explicit analytical examples of {\em exponential} mixing available in the literature, constructed in different settings.

In connection with a conjecture stated by Bressan~\cite{Bressan}, Crippa and De~Lellis~\cite{LellisCrippa} showed that if the velocity field has a uniform in time bound on the $\dot{W}^{1,p}$ norm, where $1<p\leq\infty$, then the geometric mixing scale of the solution to the continuity equation cannot decay faster than exponentially. Iyer, Kiselev and Xu~\cite{Kiselev} and Seis~\cite{Seis} later showed similar bounds for the functional mixing scale, hence
\begin{equation}
\label{lowerb}
\mathcal{G}(\rho(t,\cdot))\geq C e^{-ct}\hspace{0.5cm}\text{ and }\hspace{0.5cm}\|\rho(t,\cdot)\|_{\dot{H}^{-1}}\geq C e^{-ct} \, ,
\end{equation}
where $C>0$ and $c>0$ are constants depending on the initial datum $\rho_{0}$ and on the given bounds on the velocity field. In order to prove the sharpness of the bounds in~\eqref{lowerb}, Yao and Zlato\v{s} \cite{Yao} and Alberti, Crippa, and Mazzucato~\cite{AlbCrippaCras,AlbCrippa} constructed explicit velocity fields with the above constraints, and initial data, to which the associated solution gets mixed at an exponential rate. 

By interpolation, there is a strong connection between the decay of $\|\rho(t,\cdot)\|_{\dot{H}^{-1}}$ and the increase 
of the positive Sobolev semi-norms of $\rho(t,\cdot)$. By an iteration and scaling argument with the optimal mixer of \cite{AlbCrippa}, the authors 
of \cite{AlbCrippa2} constructed a divergence-free velocity field in $L^{\infty}(W^{1,p})$, for any given $1\leq p<\infty$, and a solution~$\rho$ to the continuity equation, so that $\rho_{0}\in C^{\infty}$ and $\rho(t,\cdot)$ does not belong to $\dot{H}^s$ for any $s>0$ and $t>0$. 

Both the example of \cite{Yao} and \cite{AlbCrippa} use a similar inductive structure for the construction. The basic idea is to equally redistribute the tracer at each step among a finer sub-grid, as schematically visualized in Figure \ref{biggie}, with tracer movements localized in the cells.
\begin{figure}[h]
	\begin{center}
		\scalebox{0.5}{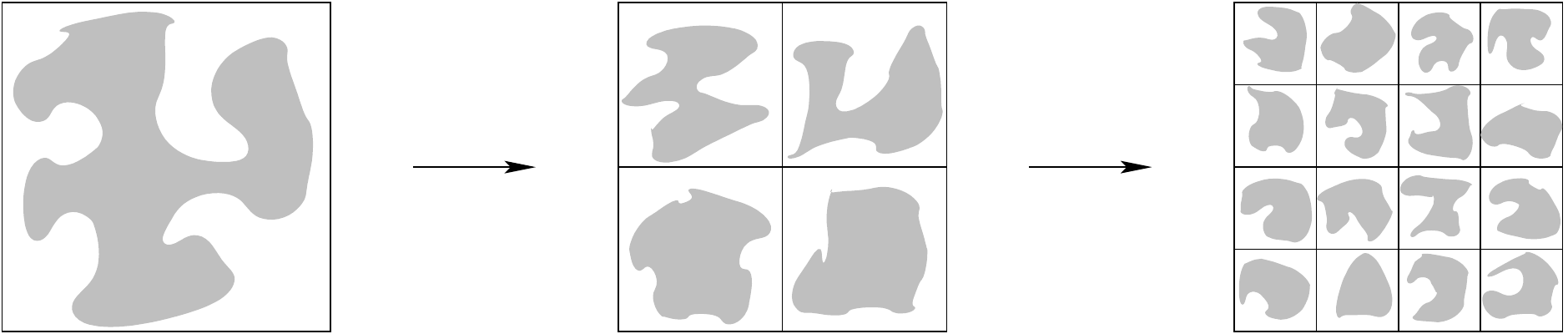}
	\end{center}
	\caption{First two steps of a cellular flow \label{biggie}}
\end{figure}
Velocity field of this so called cellular type cannot be universal mixers (see Appendix C in \cite{CellP}), which means that they 
cannot mix \emph{every} initial datum $\rho_{0}$. In fact, in the results above the velocity field only mixes a specific, 
conveniently constructed, initial datum.  Furthermore, the results in \cite{CellP} show that under a uniform-in-time bound on the $\dot{W}^{s,p}$ norm, 
where $s>1$ and $1<p\leq\infty$, any velocity field of cellular type cannot mix faster than polynomial. The numerical simulations in \cite{Lin}  suggests that exponential decay is still possible under this constraint on the norm, and hence in this case the cellular structure is responsible for slowing down the mixing process. Observe that the example in \eqref{themixa} is clearly not of cellular type. Examples of exponential universal 
mixers of non cellular type have been 
constructed in \cite{EZ}.

Other examples of mixing velocity fields were constructed in order to prove the non-uniqueness
 of solutions of the continuity equation (\cite{Depauw,Bressan,Lin}) in the case 
 where $u \notin L^{1}((0,T) ; BV)$. In this case, it is possible to have perfect mixing in finite time. By inverting time, such a perfect mixer produces 
 a non-trivial solution of the continuity equation with zero initial datum. The structure used to construct these examples is similar to the cellular type described above.
\end{remark}

\subsection*{Structure of the Paper}

The rest of the paper is organized as follows. 
In Section~\ref{pcdata} we consider a family of initial data which are piecewise constant in the radial direction, for which we can 
prove a decay of order $1/t$ for the geometric mixing scale, as well as 
a decay of order $1/\sqrt{t}$ for the functional mixing scale. This follows by a combination of the main computation 
in Lemma~\ref{rnm} and other auxiliary Lemmas in Subsections~\ref{billpu} and~\ref{sunkd}. The proofs of Theorems~\ref{MT} and \ref{ratethm} and of Proposition~\ref{rateprop} (in Subsections~\ref{merge}, \ref{takeA}, and \ref{takeB} respectively) are performed by a suitable approximation of 
different families of initial data 
with piecewise constant data, for which we can use the results of Section \ref{pcdata}. In the appendix we show that Assumption~\ref{equalradius3} is necessary for the tracer to get mixed, exploiting the role of the accuracy $\kappa$ 
in connection with the weak convergence to zero of the tracer.  

\subsection*{Acknowledgments} This research has been supported by the ERC Starting Grant 676675 FLIRT. The authors would like to thank Christian Zillinger for useful feedback and enlightening remarks on a first version of this manuscript, and for pointing out the connections with the theory in \cite{Bedrossian, Zill1, Zill2, Zeng}. The authors would also like to thank Bohan Zhou for useful remarks concerning the appendix of the manuscript.

\section{Preliminaries and the Case of Piecewise Constant Data}

Hereafter the domain of all the 
function spaces we take into account will be most of the times the two--dimensional disk $B_1$, so that in such cases we will not specify this anymore.  
For instance, we simply write $H^{s}$, $\| \cdot \|_{H^{s}}$ instead of $H^{s}(B_1)$, $\| \cdot \|_{H^{s}(B_1)}$, and so on.

\subsection{Piecewise Constant Data}\label{pcdata} 
		
		Here we first focus on a specific class of initial data, that are piecewise constant along the radial direction and satisfy Assumption~\ref{equalradius3}. More precisely, we consider 
		\begin{equation}\label{InitialDatumForm}
		\rho_{0}(r,\theta) := \sum_{\ell=0}^{2^{N}-1} \chi_{( \ell 2^{-N}, (\ell+1)2^{-N}  ]} (r) f^{\ell}(\theta), \quad N \in \mathbb{N} \cup \{0\} \, ,
		\end{equation}
		where $f^{\ell} \in L^{\infty}(\mathbb{T})$ and
		\begin{equation}\label{ZeroMeanAss}
		\int_{0}^{2\pi} f^{\ell}(\theta) d\theta =0 \, .
		\end{equation}
		 For instance, when $N=0$ and $f^{0}(\theta) := \chi_{(0,\pi]}(\theta) - \chi_{(\pi,2\pi]}(\theta)$, we are considering the 
		simple initial data
		which equals~$1$ in the upper half disk and~$-1$ in the lower half disk; see Figure \ref{GreatFigure}.

	\begin{proposition}
		\label{gridini}
		There exists an absolute constant $C$ such that the following holds. For $\rho_{0} \in L^{\infty}$ of the form~\eqref{InitialDatumForm}, we have 
	\begin{equation}\label{SpecialDataFunctGeom}
		\mathcal{G}(\rho(t,\cdot)) \leq \frac{C}{\kappa^{2} t}, \quad \text{for} \quad t \geq C \frac{2^{N}}{\kappa}\, ,
          \end{equation}
	           and
	\begin{equation}\label{SpecialDataFunct}
	     \|\rho(t,\cdot)\|_{\dot{H}^{-1}}\leq \frac{C\|\rho\|_{L^{\infty}}}{\sqrt{t}}	, \quad \text{for} \quad t \geq C2^{2N}\,.
         \end{equation}
	\end{proposition}
This proposition gives a quantitative rate of decay for both the geometric and the functional mixing scales in the case of initial data of the particular form~\eqref{InitialDatumForm}. Notice that the rate does not depend on the integer $N$ involved in the expression~\eqref{InitialDatumForm}.
        For the proof of the proposition we need some preliminary lemmas, that are also required to prove the main results in Section~\ref{Sec:Proofs}.

	\subsection{Auxiliary Lemmas} 
	
	Since the velocity field~\eqref{themixa} advects a traced point over a circle centered at zero, we will tile the unit disk with pieces of 
        annuli which behave like rectangles with bounded eccentricity. More precisely, this means that there is an absolute constant $c$ 
        such that, for any $M \in \mathbb{N}$ and 
        $Q \in \mathcal{Q}_{M}$ like below, we have that $Q$ is contained in a disk $B$ and $|Q| \geq c |B|$. Notice that 
        the area of the tiles $Q \in \mathcal{Q}_{M}$ is proportional
         to $2^{-2M}$ and their diameter is proportional to~$2^{-M}$; see Remark \ref{Rem:sizes}. 
	
	\label{billpu}
		\begin{definition}[Annular tiling]\label{Def:NewTil}
							Given any $M \in \mathbb{N} $, we tile $\overline{B_1} \setminus \{ 0 \}$ in the following way
							$$
							\overline{B_1} =\bigcup_{i=0}^{2^M-1}\bigcup_{j=0}^{i} Q_{ij}^M   \, , 
							$$  
							where $Q_{ij}^{M}$ are given, in polar coordinates, by
							\begin{equation} \nonumber
								Q_{ij}^M=\left\{ (r, \theta ) \in [0,1] \times [0,2\pi] \, : \, 
								r \in \left( i 2^{-M}, (i+1) 2^{-M} \right] , \ \theta \in 2 \pi \left( \frac{j}{i+1},\frac{j+1}{i+1}\right]  \right\} \, ,
							\end{equation}
							and we set
							$$
							\mathcal{Q}_{M} =\left\lbrace Q_{ij}^M, \textnormal{ where } i=0,\ldots,2^{M}-1 \textnormal{ and } 
							j=0,\ldots,i \right\rbrace  \, .
							$$
							\begin{figure}[h]
								\begin{center}
									\scalebox{0.35}{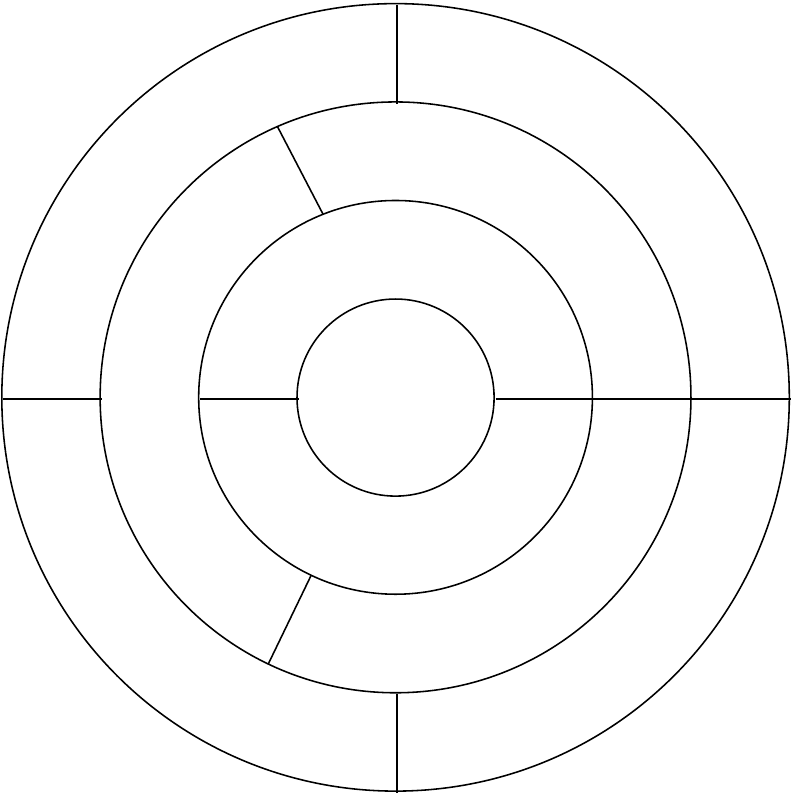}
								\end{center}
								\caption{Example of an annular tiling for $M=2$}
								\label{koda}
							\end{figure}
						\end{definition}
\begin{remark}\label{Rem:sizes}
Note that there exist constants $C_1$ and $C_2>0$, such that 
\begin{equation}
C_12^{-2M}\leq\left|Q_{ij}^M\right|\leq C_22^{-2M} \hspace{0.5cm}\text{ and }\hspace{0.5cm}C_12^{-M}\leq \diam{Q_{ij}^M}\leq C_22^{-M}
\end{equation}
for all $M\in\mathbb{N}$.
\end{remark}
The following is a slightly different version of Lemma 3.5 in \cite{AlbCrippa}. We prove that it suffices to show that the tracer is well mixed on 
all the annular tiles of diameter~$2^{-M}$ in order to show that it is well mixed on any disk of comparable diameter. 

	\begin{lemma}\label{mpc}
		Let $\rho$ be a bounded function supported in $\overline{B_1}$. If 
		\begin{equation}
		\Big|\,\fint_{Q}\rho  \, \Big| \leq \frac{\kappa}{2} \| \rho  \|_{L^{\infty}}, \quad \mbox{for all} \quad Q \in \mathcal{Q}_{M} \, ,
		\end{equation}
		then there is an absolute constant $C >0$ so that
		\begin{equation}\nonumber
		\mathcal{G}(\rho)\leq C  \frac{2^{-M}}{\kappa}  \, .   
	         \end{equation}
        \end{lemma}

	\begin{proof}
		We have
		\begin{equation}\nonumber
		\begin{split}
		\Big|\,\fint_{B_{\varepsilon}(x)}\!\!\! \rho \, \Big|&\! \leq \! \frac{1}{\pi\varepsilon^2}
		 \sum_{ \substack{ Q \in \mathcal{Q}_{M} \\ Q \subset B_{\varepsilon}(x)} } \! \Big|\,\int_{Q}\!\!\! \rho \, \Big|
		+  \frac{1}{\pi\varepsilon^2}  \!\!\!\!\!  
		\sum_{\substack{ Q \in \mathcal{Q}_{M} \\ Q \cap \partial B_{\varepsilon}(x) \neq \emptyset}}\int_{Q}\!\! \left| \rho \right|   \\
		& \leq \frac{\| \rho \|_{L^{\infty}}}{\pi\varepsilon^2} \left(  \frac{\kappa}{2}|B_{\varepsilon}(x)|+ 
		\big| B_{\varepsilon + c 2^{-M}}(x)\setminus B_{\varepsilon - c 2^{-M}}(x)   \big|  \right)   \\
		&\leq \frac{\kappa}{2} \| \rho \|_{L^{\infty}} +  \frac{4 c 2^{-M}}{\varepsilon} 
		\| \rho \|_{L^{\infty}} \leq \kappa \|  \rho  \|_{L^{\infty}} \, ,
		\end{split}
		\end{equation}
                 as long as $\varepsilon \geq  8 c 2^{-M} / \kappa $. Here $c$ is taken sufficiently large so that 
                 $\diam Q \leq c 2^{-M}$; recall Remark \ref{Rem:sizes}. 
	 \end{proof}
         A similar lemma holds also for the functional mixing scale. If the tracer is well mixed on any 
         tile $Q \in \mathcal{Q}_{M}$, then its $\dot{H}^{-1}$
         norm is small. 
	\begin{lemma}
		\label{mpc2} 
	Let $\rho$ be bounded, mean-free function supported in $\overline{B_1}$. If
		\begin{equation}\label{Ass+Poinc}
		\left|\fint_Q \rho \, \right| \leq 2 \|\rho\|_{L^{\infty}}2^{-M}, \quad \mbox{for all} \quad Q\in \mathcal{Q}_{M} \, ,
		\end{equation}
		there exists an absolute constant $C$ such that
		\begin{equation}\label{WCTPI}
		\|\rho \|_{\dot{H}^{-1}}\leq C\|\rho\|_{L^{\infty}}2^{-M} \, .
		\end{equation}
	\end{lemma}
	For the proof of Lemma \ref{mpc2} we need the following Poincar\'e estimate:
	\begin{lemma}
	[Poincar\'e inequality on tiling]
	\label{poinca}
	There exists an absolute constant $C$ such that for all $\xi \in W^{1,1}$ we have that
		\begin{equation}\label{UnitCubePoinc}
		\| \xi- \xi_Q\|_{L^{1}(Q)}\leq C2^{-M}\|\nabla \xi\|_{L^{1}(Q)}
		\end{equation}
		for any $Q\in \mathcal{Q}_M $, where $\xi_Q :=\fint_Q \xi$.
	\end{lemma}
	\begin{proof}
	First of all, since the tiles $Q_{00}^{M}$ are just disks of radius $2^{-M}$ centered at zero, \eqref{UnitCubePoinc} for $Q=Q_{00}$ is simply a rescaled version 
		of the Poincar\'e inequality on the unit disk. To handle the remaining tiles,  
		we start by the Poincar\'e inequality over a rectangle $R$ of sides $\lambda_1 \times \lambda_2$, that
		is
		\begin{equation}\label{RectanglePoincare}
		\| \xi -\xi_Q\|_{L^{1}(R)} \leq C \| (\lambda_1 \partial_{1} + \lambda_2 \partial_2) \xi \|_{L^{1}(R)}  \, ,
		\end{equation}
		which one gets by translating and rescaling the 
		Poincar\'e inequality on the unit cube. Thus, recalling 
		that $Q_{ij}^M$ has sides $2^{-M} \times \frac{2\pi}{i+1}$ (when we look at it as a rectangle in polar coordinates), we have  
		\begin{equation}
		\int_{Q_{ij}^{M}} |\xi -\xi_Q|(r,\theta) dr d\theta  
		\leq C \int_{Q_{ij}^{M}} \left| \left(2^{-M} \partial_{r} + \frac{2\pi}{i+1} \partial_\theta \right) \xi \, \right| (r,\theta) dr d\theta .
		\end{equation}
		We multiply this inequality times~$i2^{-M}$ and,
		noting that $r \simeq i2^{-M}$ when $r \in Q_{ij}^{M}$, we arrive at
		\begin{equation}
		\int_{Q_{ij}^{M}} |\xi -\xi_Q|(r,\theta) \, r \, dr d\theta  
		\lesssim C 2^{-M} \int_{Q_{ij}^{M}} \left| \left( \partial_{r} + \frac{ \partial_\theta}{r} \right) \xi \, \right| (r,\theta) \, r \, dr d\theta \, ,
		\end{equation}
		that, once we recall $\nabla = \partial_{r} + \frac{1}{r} \partial_{\theta}$, completes the proof of Lemma \ref{poinca}.
	\end{proof}
	\begin{proof}[Proof of Lemma \ref{mpc2}]	
		We work with the $\dot{H}^{-1}$ norm defined by duality as 
			\begin{equation}
			\label{tinaf}
		\|\rho\|_{\dot{H}^{-1}}=\sup\left\lbrace \int_{B_1}\rho(x)\xi(x)\,dx \,:\, \|\nabla\xi \|_{L^2}\leq 1 \right\rbrace  \, .
		\end{equation}
		First we note that there exists a constant $C>0$ such that for any mean-free function $\rho$ we have that
		\begin{equation}
		\label{pie}
		C\|\rho\|_{\dot{H}^{-1}}\leq	\|\rho\|_{H^{-1}}\leq \|\rho\|_{\dot{H}^{-1}}\,,
		\end{equation}
		where
		\begin{equation}
		\label{feyy}
		\|\rho\|_{H^{-1}}=\sup\left\lbrace \int_{B_1}\rho(x)\xi(x)\,dx \,:\, \|\xi \|_{H^1}\leq 1 \right\rbrace  \, .
        		\end{equation}
        The second inequality in \eqref{pie} is immediate. As for the first inequality, let $\xi$ such that $\|\nabla\xi \|_{L^2}\leq 1$. 
        We define $\tilde{\xi}=\xi-\xi_{B_1}$ and note that since $\rho$ is mean-free, we have that
        \begin{equation}
        \int_{B_1} \rho(x)\tilde{\xi}(x)\,dx=\int_{B_1}  \rho(x)\xi(x)\,dx \,.
        \end{equation}
        On the other hand, by the Poincaré inequality we have that
        \begin{equation}
        \|\tilde{\xi} \|_{L^2}\leq C\|\nabla\xi \|_{L^2}\hspace{0.5cm}\text{and}\hspace{0.5cm}\|\nabla\tilde{\xi} \|_{L^2}=\|\nabla\xi \|_{L^2}\,.
        \end{equation}
        By the definitions \eqref{tinaf} and \eqref{feyy}, this concludes the proof of \eqref{pie}. 
        
        In order to show \eqref{WCTPI}, by \eqref{pie} it is sufficient to show that
        	\begin{equation}
        	\label{altfacts}
        \|\rho \|_{H^{-1}}\leq C\|\rho\|_{L^{\infty}}2^{-M}\,.
        \end{equation}
        Let $\xi$ such that $\| \xi\|_{H^1}\leq 1$. Then
		\begin{align}
		\label{yolin}
		\Bigg|\int_{B_1}\rho(x)  &   \xi(x)\, dx \Bigg|  \leq \sum_{Q\in\mathcal{Q}_{M}}\left|\int_{Q}\rho(x)\xi(x)\,dx \right|\\ \nonumber
		& \leq\sum_{Q\in\mathcal{Q}_{M}}\int_{Q}\left|\rho(x)(\xi(x)-\xi_Q)\right|\,dx +\left|\int_{Q}\rho(x)\xi_Q\,dx\right|\\  \nonumber
		&\leq\|\rho\|_{L^{\infty}}\sum_{Q\in\mathcal{Q}_{M}}\|\xi-\xi_Q\|_{L^1(Q)}+\sum_{Q\in\mathcal{Q}_{M}}|\xi_Q|\left|\int_{Q}\rho(x)\,dx \right|\\  \nonumber
		&\leq C2^{-M}\|\rho\|_{L^{\infty}}\left( \sum_{Q\in\mathcal{Q}_{M}}\|\nabla \xi\|_{L^1(Q)}+
		 \sum_{Q\in\mathcal{Q}_{M}}\left| \int_Q\xi(y)\,dy\right|\right)\\ \nonumber
		&\leq C2^{-M}\|\rho\|_{L^{\infty}} \left( \|\nabla \xi\|_{L^1}+\|\xi\|_{L^1}\right)\\  \nonumber
	&\leq C2^{-M} \|\rho\|_{L^{\infty}}\|\xi \|_{H^1} \leq C\|\rho\|_{L^{\infty}}2^{-M}\, ,
		\end{align}
		where we used \eqref{UnitCubePoinc} and \eqref{Ass+Poinc} in the fourth inequality. This concludes the proof of \eqref{altfacts} and therefore of the lemma. 
	\end{proof}

The following is a key lemma that will be used, together with the subsequent one, in the proof of all the main results in the next section. Here we consider initial data   
of the form \eqref{InitialDatumForm}, namely piecewise constant along the radial direction and with zero circular mean, 
and we show that solutions are well mixed on any (small) annular tile, provided we wait a sufficiently large time. Here
we only consider tiles which are contained into the sets (annuli) where the data are radially piecewise constant. The case of large tiles, on which the data can also change their values once we move in the radial direction, will be analyzed in Lemma \ref{Wutang}.

\begin{lemma}
	\label{rnm}
	There exists an absolute constant $C >0$ such that the following holds. For any 
	$\rho_{0} \in L^{\infty}$ of the form~\eqref{InitialDatumForm} and $t \geq C \frac{ 2^{M}}{\kappa}$, we have 
	\begin{equation}
		\label{fpfand}
		\left|\fint_{Q} \rho(t, \cdot) \right|\leq \frac{\kappa}{4} \| \rho(t,\cdot)\|_{L^{\infty}}, 
		\quad \forall M > N, \quad \forall Q \in \mathcal{Q}_{M} \, . 
	\end{equation}
	If $t\geq C 2^{2M}$ we have
	\begin{equation}
		\label{fpfand2Old}
		\left|\fint_{Q}\rho(t,\cdot) \right|\leq 2^{-M}\|\rho(t,\cdot)\|_{L^{\infty}} , \quad \forall M > N, \quad \forall Q \in \mathcal{Q}_{M} \, . 
	\end{equation}
\end{lemma}
\begin{proof}
	We first prove \eqref{fpfand}. Since $Q_{00}^M=B_{2^{-M}}$, and the initial datum has zero average on any circle, a property which is preserved by the flow, we immediately have 
	\begin{equation}
		\int_{Q_{00}^M} \rho(t, \cdot)=0\, .
	\end{equation}
	Hence it is sufficient to consider $i\geq 1$. Since we are considering $\rho_{0}$ of the form~\eqref{InitialDatumForm}, the restriction of the solution~$\rho(t,\cdot)$ to  
	the tiles $Q_{ij}^{M}$ is
	$$
	\rho(t, r, \theta)\Big|_{Q_{ij}^{M}} = f^{\ell}(\theta- 2\pi t r) \, ,    
	$$ 
	where $\ell$ is the only integer such that $( i 2^{-M}, (i+1)2^{-M}] \subset ( \ell 2^{-N}, (\ell+1)2^{-N} ]$.
	We set 
	\begin{equation}\label{def:RT}
	r_{i}=i2^{-M}, \quad i=1, \ldots, 2^{M}-1 
	\qquad \text{ and } \qquad
	\theta_{j}=\frac{j}{i+1}2\pi, \quad j=0, \ldots, i \, .  
	\end{equation} 
	Let us compute
	\begin{equation}\nonumber
		\int_{Q_{ij}^{M}} \rho(t, \cdot)= \int_{\theta_{j}}^{\theta_{j+1}} \int_{r_{i}}^{r_{i+1}}  
		\rho(t, r, \theta) \, r  \, d r d\theta = \int_{\theta_{j}}^{\theta_{j+1}}  \int_{r_{i}}^{r_{i+1}} f^{\ell}(\theta-2\pi t r) \, r \, d r d\theta
		\, . 
	\end{equation}
	For any fixed $\theta$, we change variables $y= \theta - 2 \pi t r$. By $r dr = - (2\pi t)^{-2}(\theta-y)dy$, 
	we get
	\begin{equation}\nonumber
		\int_{\theta_{j}}^{\theta_{j+1}}  \int_{r_{i}}^{r_{i+1}} f^{\ell}(\theta-2\pi t r) \, r \ d r d\theta = \frac{1}{(2\pi t)^{2}} \int_{\theta_{j}}^{\theta_{j+1}} I_{1}(\theta) \ d \theta
		+
		\frac{1}{(2\pi t)^{2}} \int_{\theta_{j}}^{\theta_{j+1}} I_{2}(\theta) \ d \theta \,,
	\end{equation}
	where 
	\begin{equation}\label{kata}
		I_{1}(\theta) := -  \int_{\theta - 2\pi t r_{i+1}}^{\lceil \theta - 2\pi t r_{i+1} \rceil_{2\pi} }f^{\ell}(y) (\theta-y) \ d y 
		- \int_{\lfloor \theta - 2\pi t r_{i} \rfloor_{2\pi}}^{\theta - 2\pi t r_{i} }f^{\ell}(y) (\theta-y) \ d y \, ,
	\end{equation}
	\begin{equation}\label{yolinski}
		I_{2}(\theta) := - \sum_{k = \frac{1}{2\pi} \lceil \theta - 2\pi t r_{i+1} \rceil_{2\pi}}^{\frac{1}{2\pi} \lfloor \theta - 2\pi t r_{i} \rfloor_{2\pi}}  
		\int_{2\pi k}^{2\pi (k+1)} f^{\ell}(y)(\theta-y) \ d y \, ,
	\end{equation}
	and $\lceil a \rceil_{2\pi}$ ($\lfloor a \rfloor_{2\pi}$) is the smallest (largest) multiple 
	of $2\pi$ which is larger (smaller) than~$a$. 
	
	The integral of $I_1(\theta)$ will be small because we 
	integrate over a small set, while the integral of $I_2(\theta)$ will be small 
	due to cancellation effects arising in the integral. 
	Indeed
	$$| I_1(\theta) | \lesssim (1 + t r_{i+1}) \| f^{\ell} \|_{L^{\infty}} \leq (1 + t r_{i+1}) \| \rho_{0} \|_{L^{\infty}} \, ,$$  
	so that  
	\begin{equation}\label{turnip}
	\begin{split}
		\frac{1}{(2\pi t)^{2}} \int_{\theta_{j}}^{\theta_{j+1}} |I_{1}(\theta)|\,  d \theta 
		&\lesssim  \frac{1+t r_{i+1} }{t^{2}} (\theta_{j+1}-\theta_{j}) \| \rho_{0} \|_{L^{\infty}} \\	
		&\lesssim \left(\frac{1}{t^2}+\frac{2^{-M}}{t}\right)\| \rho(t,\cdot) \|_{L^{\infty}}  \, ,
		\end{split}
	\end{equation}
	where we have used $\theta_{j+1} - \theta_{j} = \frac{2\pi}{i+1}$. Now, recalling that
	$|Q_{ij}^{M}| \simeq 2^{-2M}$, from \eqref{turnip} we see that
	\begin{equation}\label{BTAT}
		\frac{1}{(2\pi t)^{2}} \int_{\theta_{j}}^{\theta_{j+1}} |I_{1}(\theta)|\,  d \theta	< \frac{\kappa}{8} \| \rho(t,\cdot)\|_{L^{\infty}} |Q_{ij}^{M}| 
	\end{equation}
	as long as $t \geq C \frac{2^{M}}{\kappa}$, for some large absolute constant $C$.  
	In order to estimate the contribution of $I_2(\theta)$, we notice that, since $f^{\ell}$ is $2\pi$-periodic with zero mean,
	the general term of the sum~\eqref{yolinski} reduces to
	\begin{equation}
		\int_{2\pi k}^{2\pi (k+1)} f^{\ell}(y)(\theta - y) \ d y  =  - \int_{2\pi k}^{2\pi (k+1)} f^{\ell}(y)y \ d y \, ,
	\end{equation}
	and, once we
	set $F^{\ell}(y) := \int_{0}^{y} f^{\ell} (z) dz$,
	\begin{align}\nonumber
		- \int_{2\pi k}^{2\pi (k+1)} f^{\ell}(y)y \ d y \, & = -  \Big[ F^{\ell}(y) y \Big]_{2\pi k}^{2\pi (k+1)} + \int_{2\pi k}^{2\pi (k+1)} F^{\ell}(y) \ dy
		\\  \nonumber
		&= \int_{0}^{2\pi} F^{\ell}(y) \ dy  \lesssim \| f^{\ell} \|_{L^{\infty}} \leq \| \rho_{0} \|_{L^{\infty}}    \, ,
	\end{align}
	where we have used that $F^{\ell}(0) = 0$ and that $F^{\ell}$ is $2\pi$-periodic, which follows by the fact that $f^{\ell}$ has zero mean.
	Thus, plugging 
	this into~\eqref{yolinski} and noting that there are less than~$t (r_{i+1}-r_{i}) =  t2^{-M}$ terms in the sum over~$k$,  we arrive at
	\begin{equation}
	\begin{split}
		\label{turnup}
		\frac{1}{(2\pi t)^{2}}   \int_{\theta_{j}}^{\theta_{j+1}} | I_{2}(\theta) | \ d\theta 
		& \lesssim \frac{\theta_{j+1}-\theta_{j}}{t} 2^{-M} \| \rho_{0} \|_{L^{\infty}}  \\ 
		& \simeq \frac{2^{-M}}{(i+1)t}  \| \rho(t,\cdot) \|_{L^{\infty}}
		\leq\frac{2^{-M}}{t}  \| \rho(t,\cdot) \|_{L^{\infty}} \, ,
		\end{split}
	\end{equation}
	and again, with the same computation used to deduce  \eqref{BTAT} by \eqref{turnip}, we have that
	\begin{equation}\label{BTAT2}
		\frac{1}{(2\pi t)^{2}} \int_{\theta_{j}}^{\theta_{j+1}} |I_{2}(\theta)|\,  d \theta	< \frac{\kappa}{8} \| \rho(t,\cdot) \|_{L^{\infty}} |Q_{ij}^{M}|  \, ,
	\end{equation}
	provided $t \geq C \frac{2^M}{\kappa}$, for some large constant $C$. This concludes the proof of \eqref{fpfand}. 
	Estimate \eqref{fpfand2Old} can be proved in an analogous way. Indeed, looking at \eqref{turnip} and~\eqref{turnup}, it is clear that we only 
	need to restrict to $t \geq C 2^{2M}$.	
\end{proof}


In the next lemma we show that solutions corresponding to initial data of the form \eqref{InitialDatumForm} are well mixed on (large) tiles which contain the sets (annuli) 
where the data are radially piecewise constant. This is a complement of Lemma \ref{rnm}. Notice that the estimate 
\eqref{fpfand2} below is more efficient than its counterpart \eqref{fpfand2Old}, since it even holds for smaller times.

\begin{lemma}
\label{Wutang}
There exists an absolute constant~$C$ such that the following holds.
Let $\rho_{0} \in L^{\infty}$ of the form~\eqref{InitialDatumForm}. For all $M\leq N$ we have that 
\begin{equation}
\label{fpfand2PREQ}
\left| \fint_{Q} \rho(t,\cdot) \right|\leq \frac{\kappa}{4} \|\rho(t, \cdot)\|_{L^{\infty}} 
\end{equation}
for all $Q\in\mathcal{Q}_M$ and all $t\geq C \frac{2^{N}}{\kappa}$. Similarly
\begin{equation}
\label{fpfand2}
\left|\fint_{Q}\rho(t,\cdot) \right|\leq 2^{-M}\|\rho(t, \cdot)\|_{L^{\infty}} 
\end{equation}
for all $Q\in\mathcal{Q}_M$ and all $t\geq C 2^{M+N}$.
\end{lemma}

The proof is very similar to that of Lemma \ref{rnm}. 

\begin{proof}

Given $M \leq N$, 
we define the following sub-tiling of each $Q_{ij}^{M}\in\mathcal{Q}_M$:
		\begin{equation*}
		D_{ij, M}^{k,N}=\left\{ r \in \left( i 2^{-M}+k2^{-N},  i 2^{-M}+(k+1)2^{-N} \right] , \ \theta \in 2 \pi \left( \frac{j}{i+1},\frac{j+1}{i+1}\right]  \right\}
		\end{equation*}
			for $k=0,...,2^{N-M}-1$. We denote by $\mathcal{D}_N^M$ the family of all sub-tiles $D_{ij, M}^{k,N}$.
			\begin{figure}[h]
				\begin{center}
					\scalebox{0.4}{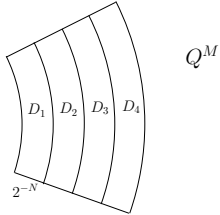}
				\end{center}
				\caption{Sub-tiling of $Q\in\mathcal{Q}^M$}
			\end{figure}
			We will show that for all $D\in \mathcal{D}_N^M$ we have that
			\begin{equation}
			\label{pluugPREQ}
	\left|\int_{D} \rho(t, \cdot)\right|\leq \frac{\kappa}{4}  |D| \,\|\rho(t, \cdot)\|_{L^{\infty}} ,  \qquad \mbox{if $t\geq \frac{C}{\kappa} 2^{N}$} \, ,  
			\end{equation}
	and
			\begin{equation}
			\label{pluug}
	\left|\int_{D} \rho(t, \cdot)\right|\leq |D|\,\|\rho(t, \cdot)\|_{L^{\infty}}2^{-M} , \qquad \mbox{if $t\geq C2^{(N+M)}$} \, .
			\end{equation}
			This is enough to prove the statement. Indeed,  this would imply,
			for any $Q\in \mathcal{Q}^M$, that
			\begin{equation*}
                      \left|\int_{Q} \rho(t, \cdot) \right|
                      \leq 
                      \sum_{\substack{D \in \mathcal{D}_N^M \\ D \subset Q} } \left|\int_{D} \rho(t, \cdot) \right| \leq \frac{\kappa}{4} \|\rho\|_{L^{\infty}} 
                     \!\!\!\!\! \sum_{\substack{D \in \mathcal{D}_N^M \\ D \subset Q} } |D|= \frac{\kappa}{4} \|\rho(t, \cdot)\|_{L^{\infty}} \left| Q\right|
			\end{equation*}
		and
        			\begin{equation*}
                      \left|\int_{Q} \rho(t, \cdot) \right|
                      \leq 
                      \sum_{\substack{D \in \mathcal{D}_N^M \\ D \subset Q} } \left|\int_{D} \rho(t, \cdot) \right| \leq  2^{-M} \|\rho\|_{L^{\infty}} 
                      \sum_{\substack{D \in \mathcal{D}_N^M \\ D \subset Q} } |D|=2^{-M} \|\rho(t, \cdot)\|_{L^{\infty}} \left| Q\right| \,,
			\end{equation*}
			which concludes the proof.  
			
			The computation of \eqref{pluugPREQ} and \eqref{pluug} is
			similar to the one we performed in the proof of Lemma~\ref{rnm}, so that we will omit 
			the redundant details. 
			We fix $D_{ij, M}^{k,N}\in\mathcal{D}_N^M$.
			Recalling the zero average condition on circles, it is sufficient to consider $i \geq 1$. 
			We let 
			$$
			r_i^k=i2^{-M}+k2^{-N}, \quad i = 1, \ldots, 2^{M}-1, \quad k=0, \ldots, 2^{N-M}-1 \, ,
			$$
			$$
			\theta_{j}=\frac{j}{i+1}2\pi, \quad j=0, \ldots, i \, .
			$$
		Proceeding as in the proof of Lemma~\ref{rnm},
		we compute
		\begin{align}\nonumber
		\int_{D_{ij, M}^{k,N}} \rho(t, \cdot) &= \int_{\theta_{j}}^{\theta_{j+1}} \int_{r_{i}^{k}}^{r_{i+1}^{k}}  
		\rho(t, \theta, r) \, r  \ d r d\theta 
		\\ \nonumber
                  & 
                   = \frac{1}{(2\pi t)^{2}} \int_{\theta_{j}}^{\theta_{j+1}}  I_{1}(\theta) \ d \theta
		+
		\frac{1}{(2\pi t)^{2}} \int_{\theta_{j}}^{\theta_{j+1}}  I_{2}(\theta) \ d \theta
		\end{align}
		where $I_{1}(\theta)$ and $I_{2}(\theta)$ are defined like in \eqref{kata} and \eqref{yolinski}, replacing $r_{i}, r_{i+1}$ with~$r_{i}^{k}, r_{i+1}^{k}$.
		%
		%
		%
		The estimate of the contribution of $I_{1}$ is the same, namely
		\begin{equation}\label{turnip2}
		\begin{split}
		\frac{1}{(2\pi t)^{2}} \int_{\theta_{j}}^{\theta_{j+1}} |I_{1}(\theta)|\,  d \theta 
		&\lesssim \left(\frac{1}{t^2}+\frac{2^{-M}}{t}\right)\| \rho_{0} \|_{L^{\infty}}  \, ,
		\end{split}
		\end{equation}
		and, by
		$|D_{ij, M}^{k,N}| \simeq 2^{-(N+M)} $, we have that
		\begin{equation}\label{grday2PREQ}
		\frac{1}{(2\pi t)^{2}} \int_{\theta_{j}}^{\theta_{j+1}} |I_{1}(\theta)|\,  d \theta	< \frac{\kappa}{8} \| \rho(t,\cdot)\|_{L^{\infty}} |D_{ij, M}^{k,N}|  \, ,
		\end{equation}
                 as long as $t \geq \frac{C}{\kappa} 2^{N} $, 
                 for some
                 absolute large constant~$C > 1$,
		and
		\begin{equation}\label{grday2}
		\frac{1}{(2\pi t)^{2}} \int_{\theta_{j}}^{\theta_{j+1}} |I_{1}(\theta)|\,  d \theta	<  \| \rho(t,\cdot)\|_{L^{\infty}} |D_{ij, M}^{k,N}| 2^{-M} \, .
		\end{equation}
		as long as $t\geq C 2^{N+M}$.
		
		The contribution of $I_2(\theta)$ is different,
		%
		%
		since in the (analogous of) the sum in~\eqref{yolinski} there are now less than~$2\pi t (r_{i+1}^{k}-r_{i}^{k}) = 2\pi t2^{-N}$ terms, so
		that we get
		\begin{equation}
		\begin{split}
		\frac{1}{(2\pi t)^{2}}  \left| \int_{\theta_{j}}^{\theta_{j+1}} I_{2}(\theta) \ d\theta \right|
		&
		\lesssim \frac{\theta_{j+1}-\theta_{j}}{t} 2^{-N} \| \rho_{0} \|_{L^{\infty}} 
		\\ 
		& = \frac{2^{-N}}{(i+1)t}  \| \rho(t,\cdot) \|_{L^{\infty}}
		\leq\frac{2^{-N}}{t}  \| \rho(t,\cdot) \|_{L^{\infty}} \, .
		\end{split}
		\end{equation}
		Again, by $|D_{ij, M}^{k,N}|  \simeq 2^{-(M+N)}$, we have that
		\begin{equation}\label{greatdayPREQ}
		\frac{1}{(2\pi t)^{2}} \int_{\theta_{j}}^{\theta_{j+1}} |I_{2}(\theta)|\,  d \theta	< \frac{\kappa}{8} \| \rho(t,\cdot)\|_{L^{\infty}} |D_{ij, M}^{k,N}|  \, ,
		\end{equation}
		as long as $t \geq \frac{C}{\kappa} 2^{N}$, and
		\begin{equation}\label{greatday}
		\frac{1}{(2\pi t)^{2}} \int_{\theta_{j}}^{\theta_{j+1}} |I_{2}(\theta)|\,  d \theta	<  \| \rho(t,\cdot) \|_{L^{\infty}} |D_{ij, M}^{k,N}| 2^{-M} \, ,
		\end{equation}
		provided $t\geq C2^{M+N}$. 
		Combining \eqref{grday2PREQ} and  \eqref{greatdayPREQ}, we arrive at \eqref{pluugPREQ},
		and
		combining \eqref{grday2} and \eqref{greatday}, we arrive at \eqref{pluug}. This concludes the proof of Lemma~\ref{Wutang}.
		\end{proof}

	\subsection{Proof of Proposition~\ref{gridini}}\label{sunkd}
	
We rely on Lemma~\ref{rnm}, in which we established how the average of the solution decays on
each annular tile. Given any $t > 2C \frac{2^{N}}{\kappa}$, 
where~$C$ is the absolute constant of the lemma,  
 we set  
 $M := \lfloor \log_{2} ( C^{-1} \kappa t ) \rfloor$.
We note that~$M > N $, by~$t > 2 C \frac{2^{N}}{\kappa}$.
By definition of~$M$, we also have~$ t \geq  C \frac{2^{M}}{\kappa}$, 
so that we can apply Lemma~\ref{rnm} and
$$
\left|\fint_{Q} \rho(t, \cdot) \right|\leq \frac{\kappa}{4} \| \rho(t,\cdot)\|_{L^{\infty}}, \quad
\forall Q \in \mathcal{Q}_{M} \, , 
$$
and then Lemma~\ref{mpc} implies
$$\mathcal{G}(\rho(t,\cdot)) \lesssim \frac{2^{-M}}{\kappa} \, .$$
Noting~$ 2^{-M}  \leq \frac{2C}{\kappa  t} $, again by definition 
of~$M$, we arrive at 
$$
\mathcal{G}(\rho(t,\cdot)) \lesssim  \frac{1}{\kappa^{2} t} \, ,$$ 
as claimed in \eqref{SpecialDataFunctGeom}.

The proof of~\eqref{SpecialDataFunct} is similar. Given $t \geq 4C2^{2N}$, with $C$ the large constant of Lemma~\ref{rnm}, we 
set $M := \lfloor \frac{1}{2} \log_{2} C^{-1}  t \rfloor$. Again we have $M > N$ and $t \geq C2^{2M}$, so that, applying the lemma 
we get
\begin{equation}
\left|\fint_{Q}\rho(t,\cdot) \right| \leq 2^{-M} \|\rho(t,\cdot)\|_{L^{\infty}}, \quad \forall Q \in \mathcal{Q}_{M} \, ,
\end{equation}
which implies by Lemma~\ref{mpc2} that
$$ 
 \|\rho (t,\cdot) \|_{\dot{H}^{-1}}\leq C\|\rho(t,\cdot)\|_{L^{\infty}}2^{-M} \, .
$$
Now, using $2^{-M} \lesssim \frac{1}{\sqrt{t}}$, the inequality~\eqref{SpecialDataFunct} follows and the proof is concluded. \hfill $\Box$

\section{Proof Theorems \ref{MT} and \ref{ratethm} and of Proposition \ref{rateprop}}\label{Sec:Proofs}

The key point in all the proofs in this section is to approximate $\rho_0$ by a sequence of piecewise constant data $\rho_{0}^{N}$ of the form \eqref{InitialDatumForm}, for which we 
have already proved decay estimates for both the geometric and functional mixing scales. The quantification of the decay of the mixing scale will turn out to strongly depend on the quantification of the approximation of the initial datum.

The approximated data $\rho_{0}^{N}$ are defined, on each $Q_{ij}^{N} \in \mathcal{Q}_{N}$, in the following way:  
\begin{equation}\label{Rho0Approx}
\rho_{0}^{N} \Big|_{Q_{ij}^{N}} = (\rho_{0})_{Q_{ij}^{N}} :=  \fint_{ Q_{ij}^{N}  } \rho_{0} \, .
\end{equation}
Note that $\rho_{0}^{N}$ satisfies Assumption~\ref{equalradius3} provided $\rho_0$ satisfies it. 
Indeed, if we take $r \in ( r_{i}, r_{i+1}]$, we have
\begin{align}\label{SeroMeanCirslesPract}
 \int_{\partial B_{r}}   \rho_{0}^{N}\,dS_r  & = 
\frac{2\pi r}{i+1} \sum_{j=0}^{i}  \rho_{0}  \Big|_{Q_{ij}^{N}}
= \frac{2\pi r}{i+1} \sum_{j=0}^{i} \int_{\theta_{j}}^{\theta_{j+1}} \int_{r_{i}}^{r_{i+1}} \rho_{0}(\theta, R) R \, dR d \theta 
\\ \nonumber
&    
= \frac{2\pi r}{i+1} \int_{r_{i}}^{r_{i+1}} \left(  \int_{\theta_{0} = 0}^{\theta_{i+1} = 2 \pi}  \rho_{0}(\theta, R) d \theta \right)  R \,  dR  = 0 \, ,
\end{align}
where, in the last identity, we have used that $\rho_{0}(\cdot, R)$ has zero average on $[0,2\pi]$ for almost every $R$; see \eqref{ZeroMeanAss}.

\subsection{Proof of Theorem~\ref{MT}}
\label{merge}

Recalling that $\mathcal{Q}_{N}$ is a family of sets of bounded eccentricity, by the Lebesgue Differentiation 
Theorem and Dominated Convergence Theorem we have that
\begin{equation}
\label{9th}
\lim_{N\to \infty} \|\rho_0^N-\rho_0\|_{L^1}\to 0\, .
\end{equation}
Now let $M\in\mathbb{N}$ be fixed. By \eqref{9th}, we can choose $N$ large enough, so that 
\begin{equation}
\label{trapin}
\|\rho_0^N-\rho_0\|_{L^1}\leq 2^{-2M}\left(\frac{\kappa}{4}\right)^2\|\rho_0\|_{L^{\infty}}\, .
\end{equation}
Denoting $\rho^N(t,\cdot)$ the evolution of $\rho_{0}^{N}$ at time $t$, we define the set
\begin{equation}\label{TCHEB1}
A_t^N=\left\lbrace | \rho^N(t,\cdot)-\rho(t,\cdot)| >\frac{\kappa}{4}\|\rho_0\|_{L^{\infty}} \right\rbrace \, .
\end{equation}
Using that the flow is measure preserving, by 
Chebychev inequality and~\eqref{trapin} we have
\begin{equation}\label{TCHEB2}
| A_t^N | = |A_0^N| 
\leq  \frac{\| \rho_{0}^{N} - \rho_{0}\|_{L^{1}}}{\frac{\kappa}{4}\|\rho_0\|_{L^{\infty}} }   \leq 2^{-2M}\left(\frac{\kappa}{4}\right) \, .
\end{equation}
We decompose
\begin{equation}
\label{dra}
\left|\fint_{B_{2^{-M}}(x)} \rho(t,\cdot)
\right|\leq \left|\fint_{B_{2^{-M}}(x)} (\rho - \rho^N)(t,\cdot) \right|
+ \left|\fint_{B_{2^{-M}}(x)}\rho^N (t,\cdot) \right| \, .
\end{equation}
Notice that, as a consequence of Proposition~\ref{gridini} with an accuracy parameter $\kappa/2$, the second term on the right is 
bounded by~$\frac{\kappa}{2} \|\rho(t,\cdot)\|_{L^{\infty}}$, for all sufficiently 
large $t$.
%
%
For the first term we can bound
\begin{equation}
\begin{split}
\left|\fint_{B_{2^{-M}}(x)}(\rho-\rho^N)(t,\cdot) \right|& \leq \left|\frac{1}{\pi 2^{-2M}}\int_{B_{2^{-M}}(x)\cap A_t^N}(\rho-\rho^N)(t,\cdot) \right|\\
& \qquad \qquad +\left|\frac{1}{\pi 2^{-2M}}\int_{B_{2^{-M}}(x) \setminus A_t^N}(\rho-\rho^N)(t,\cdot) \right|\\
& \leq \frac{1}{\pi 2^{-2M}} 2 \|\rho\|_{L^{\infty}} |A_{t}^N| + \sup_{y \notin A_{t}^N}(\rho-\rho^N)(t,y)   \\
& \leq \frac{\kappa}{4} \|\rho\|_{L^{\infty}}  +  \frac{\kappa}{4}\|\rho\|_{L^{\infty}}
\leq \frac{\kappa}{2}\|\rho\|_{L^{\infty}}
\end{split}
\end{equation}
where in the last inequality we have used \eqref{TCHEB2} and~\eqref{TCHEB1}. Back to \eqref{dra}, we have shown that
\begin{equation}
\left|\fint_{B_{2^{-M}}(x)} \rho(t,\cdot) \right|\leq\kappa\|\rho\|_{L^{\infty}}
\end{equation}
for all sufficiently large $t$. Since $M$ was arbitrary, we conclude that $\mathcal{G}(\rho(t,\cdot))\to 0$ as $t \to \infty$. The proof for the $\dot{H}^{-1}$ norm is similar. An analogous argument shows that for any $M\in \mathbb{N}$ and $t$ sufficiently large, we have that
	\begin{equation}
	\fint_Q \rho(t, \cdot) \leq 2C\|\rho\|_{L^{\infty}}2^{-M}
	\end{equation}
	for any $Q\in \mathcal{Q}_{M}$, which implies, by Lemma \ref{mpc2}, that 
	$$
	\|\rho(t,\cdot)\|_{\dot{H}^{-1}}\leq C\|\rho\|_{L^{\infty}}2^{-M}
	$$ 
	for $t$ sufficiently large. Since $M$ is arbitrary, we conclude that $\|\rho(t,\cdot)\|_{\dot{H}^{-1}}\to 0$ as $t \to \infty$. \hfill $\Box$

\subsection{Proof of Theorem \ref{ratethm}}\label{takeA} 

(i) If $\rho_{0}$ is continuous on $\overline{B_1}$ and zero on $\R^{2} \setminus \overline{B_1}$, we can choose $N$ sufficiently large so that  
\begin{equation}
\label{tasta}
\| \rho(t,\cdot) - \rho^{N}(t, \cdot)\|_{L^{\infty}}
=
\| \rho_{0} - \rho^{N}_{0}\|_{L^{\infty}} 
< \frac{\kappa}{4} \| \rho_{0} \|_{L^{\infty}}
= \frac{\kappa}{4} \| \rho(t,\cdot) \|_{L^{\infty}} \, .
\end{equation}
Now, for all $t \geq C 2^{N+3} / \kappa$ we set
 \begin{equation}\label{WITSEWHAUD}
 M = \lfloor \log_{2} (C^{-1} \kappa t) \rfloor  \, ,
 \end{equation}
 where $C$ is the constant in Lemma~\ref{rnm}.
This implies 
$M > N$ and~$t \geq C 2^{M} / \kappa$. Notice that
for any $Q_{ij}^{M}\in\mathcal{Q}_M$ we have 
\begin{equation}
\left| \fint_{Q_{ij}^{M}} \rho(t,\cdot) \right|\leq\left| \fint_{Q_{ij}^{M}} \rho(t,\cdot) - \rho^{N}(t, \cdot) \right|+\left| \fint_{Q_{ij}^{M}} \rho^N(t,\cdot) \right| \, . 
\end{equation}
By \eqref{tasta}, the first term is bounded by $\frac{\kappa}{4} \| \rho(t,\cdot) \|_{L^{\infty}}$. Using Lemma~\ref{rnm}, 
the second term is also bounded by $\frac{\kappa}{4} \| \rho(t,\cdot) \|_{L^{\infty}}$.
Recalling Lemma~\ref{mpc}, this gives 
		\begin{equation}\nonumber
		\mathcal{G}(\rho(t,\cdot))\leq C  \frac{2^{-M}}{\kappa}  \leq  \frac{2 C^{2}}{\kappa^{2} t} \, ,   
		\end{equation}
with a possibly larger constant $C$,
where in the second estimate we have again used definition~\eqref{WITSEWHAUD}. This concludes the proof of~\eqref{ContRes}.

\medskip

(ii) We now let $\rho_{0}$ belong to $\dot{W}^{\alpha,1}$, for some $\alpha \in (0,1]$. We begin by proving the following inequalities. 
\begin{claim}
	\label{maloney}
	Let $\alpha \in (0,1)$. Then there exists a constant $C=C(\alpha)$ such that for all $N\in\mathbb{N}$ and $Q\in\mathcal{Q}_{N}$ we have that
	\begin{equation}\label{maloney2}
	\|\rho_0-(\rho_0)_Q\|_{L^1(Q)}\leq C2^{-N\alpha}\iint\limits_{Q\times Q}\frac{|\rho_0(x)-\rho_0(y)|}{|x-y|^{2+\alpha}}\,dx\,dy\, ,
	\end{equation}
	and there exists a constant $C$ such that for all $N\in\mathbb{N}$ and $Q\in\mathcal{Q}_{N}$ we have that
		\begin{equation}\label{PoincRescaled}
		\|\rho_0-(\rho_0)_Q\|_{L^1(Q)} \leq C2^{-N}\|\nabla\rho_0\|_{L^1(Q)} \, .
		\end{equation}
\end{claim}

\begin{proof}
The family of Poincar\'e inequalities~\eqref{PoincRescaled} has been already proved in Lemma~\ref{mpc2}. 
In order to prove~\eqref{maloney2} we recall that $|Q| \simeq 2^{-2N}$ and $\diam Q \simeq 2^{-N}$, so that
$$
\frac{1}{|Q|} \lesssim \frac{2^{-N\alpha}}{|x-y|^{2+\alpha}}, \quad \mbox{if} \quad (x,y) \in Q \times Q \, .
$$
Therefore we have
\begin{align}
\int\limits_Q |\rho_0(y)-(\rho_0)_Q|\,dy 
& 
= \int\limits_Q \left|\rho_0(y)- \fint \rho_0 (x) dx \right|\,dy
\\ \nonumber
& \leq \frac{1}{|Q|}\int\limits_Q  \int\limits_Q | \rho_0(y) -  \rho_0 (x) |\, dx dy 
\\ \nonumber
& \lesssim 2^{-N\alpha} \int\limits_Q  \int\limits_Q \frac{| \rho_0(y) -  \rho_0 (x) |}{|x-y|^{2+\alpha}} \, dx dy \, . \qedhere 
\end{align} 
\end{proof}

As a consequence of Claim \ref{maloney}, we compute the following rate of the approximation
for the initial data.
\begin{claim}
	\label{jodyhi}
	For any $\alpha\in(0,1]$, there exists a constant $C=C(\alpha)$ such that for any $N\in \mathbb{N}$ we have that
\begin{equation}
\|\rho_0-\rho_0^N\|_{L^1}\leq C2^{-N\alpha}\|\rho_0\|_{\dot{W}^{\alpha,1}} \, .
\end{equation}
\end{claim}
\begin{proof}
Let first $\alpha=1$. Recalling the rescaled Poincar\'e inequality \eqref{PoincRescaled}, we can compute
\begin{equation}
\begin{split}
\|\rho_0-\rho_0^N\|_{L^1}
& =\sum_{Q\in\mathcal{Q}_N}\|\rho_0-(\rho_0)_Q\|_{L^1(Q)}
\\
&\leq C2^{-N}\sum_{Q\in\mathcal{Q}_N}\|\nabla\rho_0\|_{L^1(Q)}
=C2^{-N}\|\rho_0\|_{\dot{W}^{1,1}}\,.
\end{split}
\end{equation}
Similarly, for $\alpha\in(0,1)$ we compute
\begin{align}
\|\rho_0-\rho_0^N\|_{L^1}&=\sum_{Q\in\mathcal{Q}_N}\|\rho_0-(\rho_0)_Q\|_{L^1(Q)} \nonumber \\
&\leq C2^{-N\alpha}\sum_{Q\in\mathcal{Q}_N}\,\iint\limits_{Q\times Q}\frac{|\rho_0(x)-\rho_0(y)|}{|x-y|^{2+\alpha}}\,dx\,dy  \label{OnesIn1} \\
&\leq C2^{-N\alpha}\iint\limits_{B_1 \times B_1}\frac{|\rho_0(x)-\rho_0(y)|}{|x-y|^{2+\alpha}}\,dx\,dy \nonumber \\
&=C2^{-N\alpha}\|\rho_0\|_{\dot{W}^{\alpha,1}} \, . \qedhere 
\end{align}
\end{proof}

\medskip

We can now go back to the proof of Theorem \ref{ratethm}(ii). 
We choose 
$$
c=c(\kappa, \|\rho_0\|_{\dot{W}^{\alpha,1}}, \|\rho_0\|_{L^{\infty}})
$$ 
sufficiently large, in such a way that
\begin{equation}\label{TchRhs}
\|\rho_0\|_{\dot{W}^{\alpha,1}}\leq \frac{2^{c}}{2C^{2}}\left(\frac{\kappa}{8}\right)^2\|\rho_0\|_{L^{\infty}}\, , 
\end{equation}
where the constant $C$ is larger than the one in \eqref{OnesIn1}, twice the one in Lemma~\ref{Wutang}, and 
such that $|Q|\geq \frac{1}{C} 2^{-2M}$, for all $M\in\mathbb{N}$ and all $Q\in\mathcal{Q}_M$.
Given any
\begin{equation}\label{SobT}
t \geq \frac{C}{\kappa} 2^{\frac{2+c}{\alpha}} \, ,
\end{equation}
we set 
\begin{equation}\label{SobM}
 M := \left\lfloor  \log_2 \left(\frac{\kappa}{C}\right)^{\frac{\alpha}{2}} 2^{-\frac{c}{2}} t^{\frac{\alpha}{2}}  \right\rfloor \, .
\end{equation}
Notice that by \eqref{SobT} we have $M \geq 1$. We define
\begin{equation}\label{SobSigmaM}
\sigma(M):=\left\lceil\frac{2M+c}{\alpha} \right\rceil 
\end{equation}
and notice that $\sigma(M) > M$. By \eqref{SobM} we have
\begin{equation}\nonumber
t \geq \frac{C}{2\kappa} 2^{\sigma(M)} \, ,
\end{equation}
so that we are allowed to apply Lemma~\ref{Wutang} to the solution $\rho^{\sigma(M)}$. 
Let
\begin{equation}
A_t^{\sigma(M)}=\left\lbrace |\rho^{\sigma(M)}(t,\cdot)-\rho(t,\cdot)|>\frac{\kappa}{8}\|\rho_0\|_{L^{\infty}} \right\rbrace\, .
\end{equation}
Using Claim \ref{jodyhi} and \eqref{TchRhs}, we have that
\begin{equation}
\label{trapin2}
\|\rho_0^{ \sigma(M) }-\rho_0\|_{L^1}\leq \frac{1}{2C} 2^{-2M}\left(\frac{\kappa}{8}\right)^2\|\rho_0\|_{L^{\infty}} \, ,
\end{equation}
for any $M\in\mathbb{N}$. This implies, via  Chebychev's inequality, that
\begin{equation}
\label{yoda}
\left|A_t^{\sigma(M)}\right|=\left|\left\lbrace |\rho_0^{\sigma(M)}-\rho_0|>\frac{\kappa}{8} \|\rho_0\|_{L^{\infty}} \right\rbrace\right|\leq \frac{1}{2C}2^{-2M}\left(\frac{\kappa}{8}\right) \, .
\end{equation}
Let $Q\in\mathcal{Q}_M$. We have 
\begin{equation}
\label{riffrif}
\left|\fint_{Q} \rho(t,\cdot) \right|\leq \fint_{Q}\left| \rho(t, \cdot)-\rho^{\sigma(M)} (t, \cdot) \right| + \left|\fint_{Q}\rho^{\sigma(M)} (t, \cdot) \right|\, ,
\end{equation}
and the second term on the right hand side is estimated by $\frac{\kappa}{4}\|\rho\|_{L^{\infty}}$, using 
Lemma~\ref{Wutang} (recall that $\sigma(M) > M$).
In order to bound the first term we need to use \eqref{yoda}. Indeed, recalling also that $|Q|\geq \frac{1}{C} 2^{-2M}$, we have
%
%
%
%
\begin{equation}\nonumber
\begin{split}
\fint_{Q}  \! \left| \rho(t, \cdot) \! - \! \rho^{\sigma(M)} (t, \cdot)   \right| \!
  & \leq  C 2^{2M} \!\! \int_{Q\cap A_t^{\sigma(M)}} \left| \rho(t, \cdot) \! - \! \rho^{\sigma(M)}(t, \cdot)\right| \!
\\ & \qquad \qquad \qquad
+ \! \frac{1}{|Q|}\int_{Q\setminus A_t^{\sigma(M)}} \left| \rho(t, \cdot) \! - \! \rho^{\sigma(M)} (t, \cdot) \right| \\
& \leq \! C 2^{2M} 2 \|\rho\|_{L^{\infty}} \left| A_t^{\sigma(M)} \right| + \frac{\kappa}{8}\|\rho\|_{L^{\infty}} \\
& \leq \! \frac{\kappa}{4}\|\rho\|_{L^{\infty}}\, .
\end{split}
\end{equation}

In conclusion, we have shown that the averages of $\rho^{\sigma(M)}$ over the elements of $\mathcal{Q}_M$ are bounded by $\frac{\kappa}{2}\|\rho\|_{L^{\infty}}$,
as long as $t$ satisfies \eqref{SobT}.
By Lemma \ref{mpc}, this implies that
\begin{equation}
\label{testdr}
\mathcal{G}(\rho(t,\cdot))\leq \frac{C}{\kappa}2^{-M} \, ,
\end{equation}
but, recalling \eqref{SobM}, we also have 
\begin{equation}\label{TchRhsBis} 
\frac{C}{\kappa}2^{-M} \leq \left( \frac{C}{\kappa} \right)^{1+\frac{\alpha}{2}} 2^{1+\frac{c}{2}} t^{-\frac{\alpha}{2}} \, ,
\end{equation} 
so that \eqref{SobRes} has been proved. \qed
%

\subsection{Proof of Proposition \ref{rateprop} }\label{takeB} 
(i) We are assuming that $\rho_0\in C^{0,\alpha}$, for some $\alpha \in (0,1]$. Let us start by proving the following claim.
\begin{claim}
	\label{yone}
	 We have
	\begin{equation}\label{eq:yone}
  \|\rho(t,\cdot)-\rho^N(t,\cdot)\|_{L^{\infty}} \leq C \| \rho_0\|_{C^{0,\alpha}}  2^{-N\alpha} \, ,
	\end{equation}
	for all $N\in \mathbb{N}$ and for some absolute constant $C$, where
	$$
	\|  \rho_0  \|_{C^{0,\alpha}} = \|\rho_0\|_{L^{\infty}} + \sup_{x,y \in B_1, x \neq y } \frac{|\rho_0(x) - \rho_0(y)|}{|x-y|^{\alpha}} \, .
	$$
\end{claim}

\begin{proof}
Clearly 
$$
\|\rho(t,\cdot)-\rho^N(t,\cdot)\|_{L^{\infty}} = \|\rho_{0}-\rho^N_{0}\|_{L^{\infty}} \, .
$$
On the other hand, if $x\in Q$, with $Q\in\mathcal{Q}_N$, we can bound  
 \begin{align}\nonumber
\left|\rho_0(x)-\rho_0^N(x)\right|
& \leq   
 \fint_Q\left|\rho_0(x)-\rho_0(y)\right|\,dy 
 \\ \nonumber
 &
\leq \|  \rho_0  \|_{C^{0,\alpha}}  \fint_Q\left| x-y\right|^\alpha\,dy \lesssim \|  \rho_0  \|_{C^{0,\alpha}} 2^{-N\alpha}\, ,
 \end{align}
 where we have used that $\diam (Q) \lesssim 2^{-N}$.
\end{proof}

\medskip

We can now pass to the proof of Proposition \ref{rateprop}(i). Let 
\begin{equation}\label{LongCalc}
t \geq (8C)^{\frac{\alpha+1}{\alpha}} \left( \frac{\| \rho_0\|_{C^{0,\alpha}}}{\| \rho_0\|_{L^{\infty}}}\right)^{\frac{1}{\alpha}} \, ,
\end{equation}
where $C > 1$ is larger than the constants in Lemma 
\ref{Wutang} and in~\eqref{eq:yone}. Then  
we set 
\begin{equation}\label{TAVCN}
M = \left \lfloor  \log_2 \left(  \left( \frac{\| \rho_0\|_{L^{\infty}}}{\| \rho_0\|_{C^{0,\alpha}}} \right)^{\frac{1}{\alpha+1}}  
\frac{ t^{\frac{\alpha}{\alpha+1}} }{2C} \right) \right \rfloor \, .
\end{equation}
Notice that 
\begin{equation}\label{WC2}
2^{-M} <  \left( \frac{\| \rho_0\|_{C^{0,\alpha}}}{\| \rho_0\|_{L^{\infty}}} \right)^{\frac{1}{\alpha+1}}
\frac{4C}{t^{\frac{\alpha}{\alpha+1}}} \,.
\end{equation}
Moreover, \eqref{LongCalc} ensures that $M \geq 1$. Finally we set
\begin{equation}\label{sdvb}
\sigma(M) = \left\lceil 
 \log_2  \left( 
 \left( \frac{2C \| \rho_0\|_{C^{0,\alpha}} }{\| \rho_0\|_{L^{\infty}}} 2^{M} \right)^{\frac{1}{\alpha}} \right)
\right\rceil \, .
\end{equation} 
Notice that $\sigma(M) \geq M +1$ and 
\begin{equation}\label{Claim1+This}
C 2^{-\sigma(M) \alpha} \| \rho_0\|_{C^{0,\alpha}} \leq \frac{1}{2} 2^{-M} \| \rho_0\|_{L^{\infty}} \, .
\end{equation}
Again by \eqref{TAVCN}, we see that 
\begin{equation}\label{C2Diverge}
t > 2C 2^{M} \left( 2C \frac{\| \rho_0\|_{C^{0,\alpha}}}{\| \rho_0\|_{L^{\infty}}} 2^{M} \right)^{\frac{1}{\alpha}} \, ,
\end{equation}
and, by \eqref{sdvb}, that 
$$
\left( 2C \frac{\| \rho_0\|_{C^{0,\alpha}}}{\| \rho_0\|_{L^{\infty}}} 2^{M} \right)^{\frac{1}{\alpha}} > \frac{2^{\sigma(M)}}{2} \, ,
$$
so that we have $t \geq C 2^{M + \sigma(M)}$, namely (using also $\sigma(M) \geq M +1$) we are under the assumptions of Lemma \ref{Wutang}, when we consider the 
solution $\rho^{\sigma(M)}$. 
Thus,  
for any $Q\in\mathcal{Q}_M$, we can bound
\begin{equation}
\label{barts}
\left| \fint_{Q} \rho(t,\cdot) \right|\leq\left| \fint_{Q} \rho(t,\cdot) - \rho^{\sigma(M)}(t, \cdot) \right|+\left| \fint_{Q} \rho^{\sigma(M)}(t,\cdot) \right|\, .
\end{equation}
Both the terms on the right hand side are bounded by $\frac12 2^{-M} \|\rho(t,\cdot)\|_{L^{\infty}}$, the first one because 
of Claim \ref{yone} and the inequality \eqref{Claim1+This}, the second one as consequence of Lemma \ref{Wutang}.
%
Hence the statement follows by Lemma \ref{mpc2}. Indeed, recalling also \eqref{WC2}, we arrive at
\begin{equation}\label{Finally1}
\|\rho(t,\cdot)\|_{\dot{H}^{-1}} \leq C 2^{-M} \| \rho_0 \|_{L^{\infty}}  \leq  
4 C^{2}  \| \rho_0\|_{C^{0,\alpha}}^{\frac{1}{\alpha+1}} \| \rho_0\|_{L^{\infty}}^{\frac{\alpha}{\alpha+1}} t^{-\frac{\alpha}{\alpha+1}} \, ,
\end{equation} 
so that the proof of \eqref{HoeldRes} is concluded.
%
%
%
%

\medskip

(ii) The proof is a variation of that of Theorem \ref{ratethm}(ii). Let us take 
\begin{equation}\label{Sob2T}
t \geq C 8^{\frac{\alpha+4}{\alpha}} 2^{\frac{c}{\alpha}}  \, ,
\end{equation} 
where $C$ is as before and $c$ 
is sufficiently large so that
\begin{equation}\label{TchRhsBisTris}     
\|\rho_0\|_{\dot{W}^{\alpha,1}}\leq \frac{2^{c}}{2C^{2}} \|\rho_0\|_{L^{\infty}}\, .
\end{equation}
Setting
\begin{equation}\label{Sob2Small}
M := \left\lfloor \log_2 \frac{1}{2} C^{-\frac{\alpha}{\alpha+4}} 2^{-\frac{c}{\alpha+4}} t^{\frac{\alpha}{\alpha+4}} \right\rfloor\,,
\end{equation}
a similar argument as in the proof of Theorem \ref{ratethm}(ii) shows that for 
\begin{equation}\label{Sob2Bigg}
\sigma(M):=\left\lceil\frac{4M+c}{\alpha}\right\rceil
\end{equation} we have that
\begin{equation}
\label{trapin3}
\|\rho_0^{\sigma(M)}-\rho_0\|_{L^1}\leq \frac{1}{2C} 2^{-4M}\|\rho_0\|_{L^{\infty}} \, .
\end{equation}
Thus, using
Chebychev's inequality, 
we have that
\begin{equation}
\left|A_t^{\sigma(M)}\right|=\left|\left\lbrace | \rho^{\sigma(M)}(t,\cdot)-\rho(t,\cdot)|>2^{-M}\|\rho_0\|_{L^{\infty}} \right\rbrace\right|\leq \frac{1}{2C} 2^{-3M}\, .
\end{equation}
We bound
\begin{equation}
\label{riffrif2}
\left|\fint_{Q} \rho(t, \cdot)\right|\leq \fint_{Q} \left| \rho(t, \cdot)-\rho^{\sigma(M)})(t, \cdot) \right|+ \left|\fint_{Q}\rho^{\sigma(M)} (t, \cdot) \right|
\end{equation}
for all $Q\in\mathcal{Q}_M$. 
The second term is bounded by $2^{-M}\|\rho\|_{L^{\infty}}$ by Lemma \ref{Wutang},
that we are allowed to apply beacuse $t\geq C 2^{M+\sigma(M)}$ and $\sigma(M) > M \geq 1$, looking at~\eqref{Sob2T},~\eqref{Sob2Small}, and \eqref{Sob2Bigg}. 
The same argument used in the proof of \eqref{SobRes} now shows that the second term is also bounded by $2^{-M}\|\rho\|_{L^{\infty}}$. 
In conclusion, we have shown that  the average of $\rho(t, \cdot)$ over the elements of
$\mathcal{Q}_M$ is bounded by $2\|\rho\|_{L^{\infty}}2^{-M}$.
Thus Lemma \ref{mpc2} implies that 
\begin{equation}
	\|\rho(t,\cdot)\|_{\dot{H}^{-1}} \leq C \|\rho_0\|_{L^{\infty}}2^{-M}\, , 
\end{equation}
and, noting that (see \eqref{Sob2Small})
\begin{equation}\label{TchRhsBisFour}  
2^{-M} \leq 4 C^{1+\frac{\alpha}{\alpha+4}} 2^{\frac{c}{\alpha+4}}  t^{-\frac{\alpha}{\alpha+4}} \, ,
\end{equation}
the proof of \eqref{SobRes2} is concluded.\qed

\section{Appendix: Necessity of Assumption~\ref{equalradius3} }

In Proposition~\ref{WeakDecay} we show that, if the geometric mixing scale of a solution decays to zero for any accuracy parameter $\kappa \in (0,1)$, then such a solution converges to zero weakly in $L^2$. This would not be the case just assuming decay for a given fixed $\kappa$, as pointed out in Remark~\ref{Pato}. 
This fact is then used in Proposition~\ref{NecCond} to show that the zero average condition of Assumption~\ref{equalradius3} is necessary for any bounded initial density in order to get mixed (in either geometric or functional sense) by the velocity field $u$. 

\begin{proposition}\label{WeakDecay}
Let $\rho_{0} \in L^{\infty}$ supported in $\overline{B_{1}}$ be an initial datum for which 
$\lim_{t \to \infty} \mathcal{G}(\rho(t, \cdot)) \to 0$ for all $\kappa \in (0,1)$. Then $\rho(t, \cdot)$ converges to zero weakly in $L^{2}$ as $t \to \infty$.
\end{proposition} 

\begin{remark} 
Notice that in the above proposition we are using neither the precise form of the velocity field, nor the fact that the domain is the unit disk. This is a general result relating the decay to zero of the geometric mixing scale to the weak convergence to zero.
\end{remark}

\begin{proof}[Proof of Proposition~\ref{WeakDecay}]
By the density of continuous functions in $L^{2}$ it suffices to show that we have 
$$
\lim_{t \to \infty }\int_{B_1} \rho(t, x) \phi(x) \ dx = 0 \, , 
$$
for all $\phi \in C(\overline{B_1})$. Let $\delta>0$ be given. Our goal is to show that there exists a time $t_0$ such that
\begin{equation}
\label{waltzbill}
\left|\int_{B_1} \rho(t, x) \phi(x) \ dx\right|\leq \delta
\end{equation}
for all $t\geq t_0$. Since $\phi$ is continuous there exists $\bar{\epsilon}=\bar{\epsilon}(\delta,\phi)$ such
that $|\phi(x)-\phi(y)|\leq \frac{\delta}{3}(\|\rho\|_{L^\infty}\pi)^{-1}$ for all $y\in B_\epsilon(x)$ for all $0<\epsilon\leq \bar{\epsilon}$. 
Furthermore, we choose a finite family of disjoint disks $B^1,\ldots, B^N$ such that
\begin{equation}
\label{iiro}
B^i\subset B_1,\hspace{0.5cm}\diam B^i\leq \bar{\epsilon},\hspace{0.5cm}\text{and }\,\left| A \right|\leq \frac{\delta}{3}\left(\|\rho\|_{L^\infty}\|\phi\|_{L^\infty}\right)^{-1} \,,
\end{equation}
where $A=B_1\setminus\bigcup_{i=1}^{N} B^i$. Let $x_i$ be the centers of the disks $B^i$. We have that
\begin{equation}
\label{splitter}
\begin{split}
\left|  \int_{B_1} \rho(t, x)  \phi(x) dx \right| 
& \leq \sum_{i=1}^N \left| \phi(x_{i}) \right| \left| \int_{B^i} \rho(t, x) dx  \right| \\
& \qquad + \sum_{i=1}^N
\max_{x \in B^i} \left|  \phi(x)  - \phi(x_i)  \right|   \int_{B^i} \left| \rho(t, x) \right| dx \\
& \qquad + \int_A \left|\phi(x)\rho(t,x)\right|\,dx \\
& = I+II+III\,.
\end{split}
\end{equation}
Since $\lim_{t \to \infty} \mathcal{G}(\rho(t, \cdot)) \to 0$ for all $\kappa \in (0,1)$, 
using the forthcoming Lemma \ref{Lemma:AllBiggerTimes} with $r = r_i = \frac12 \diam B_i$,
taking $t_0$ the maximum of the $\bar{t}(r_i, \cdot)$,
and choosing 
$$
\kappa=\frac{\delta}{3}\left(\|\rho\|_{L^\infty}\sum_{i=1}^{N}\left|\phi(x_i)\right|\cdot\left|B^i\right|\right)^{-1}
$$ 
we have that, for all $t\geq t_0$
\begin{equation}
\label{tearzfor}
I=\sum_{i=1}^N \left| \phi(x_{i}) \right| \left| \int_{B^i} \rho(t, x) dx  \right|\leq \sum_{i=1}^N \left| \phi(x_{i}) \right|\cdot \left|B^i\right|\|\rho\|_{L^\infty}\kappa\leq \frac{\delta}{3}\,.
\end{equation} 
For the second term we have that
\begin{equation}\label{e:anche}
II\leq \frac{\delta}{3}(\|\rho\|_{L^\infty}\pi)^{-1}\sum\limits_{i=1}^{N}\int_{B^i} \left| \rho(t, x) \right| dx\leq \frac{\delta}{3}\,.
\end{equation}
Finally, by \eqref{iiro} we can estimate
\begin{equation}
\label{esbjorn}
III\leq \left|A\right|\|\rho\|_{L^\infty}\|\phi\|_{L^\infty}\leq \frac{\delta}{3}\,.
\end{equation}
Hence combining \eqref{splitter} with equations \eqref{tearzfor}, \eqref{e:anche}, and~\eqref{esbjorn}, we have shown \eqref{waltzbill}, which completes the proof of Proposition \ref{WeakDecay}.
\end{proof}

\begin{lemma}\label{Lemma:AllBiggerTimes}
		Let $\rho_{0} \in L^{\infty}$ supported in $\overline{B_{1}}$ be an initial 
		datum for which we have~$\lim_{t \to \infty} \mathcal{G}(\rho(t, \cdot)) \to 0$ for all $\kappa \in (0,1)$.
		Then $\forall r>0,\kappa>0$ there exists a 
		time $\bar{t} = \bar{t}(r, \kappa)$ such that
		\begin{equation}
		\label{resu}
		\left|\fint_{B_r(x)}\rho(t,y)\, dy\right|\leq\kappa \|\rho\|_{\infty}
		\end{equation}
		for all $x$ and all $t\geq \bar{t}$.
	\end{lemma}
	\begin{proof}
		Fix $\kappa>0$ and $r>0$. Our goal is to find a time $\bar{t}$ such that \eqref{resu} holds. Since for $\kappa'=\kappa/2$ we have that $\mathcal{G}_{\kappa'}(\rho(t,\cdot))\to 0$, there exists a time $\bar{t}$, such that for all $t\geq \bar{t}$ there exists a radius $\delta=\delta(t)\leq \frac{\kappa}{12}r$ such that
		\begin{equation}\label{UsingAlsoThisOne}
		\left|\fint_{B_{\delta(t)}(x)}\rho(t,y)\, dy\right|<\frac{\kappa}{2} \|\rho\|_{\infty}
		\end{equation}
		for all $x$. For an arbitrary $x$ and $t\geq \bar{t}$ we have (hereafter $\delta=\delta(t)$)
		\begin{equation}
		\begin{split}
		\int_{B_r(x)}\fint_{B_{\delta}(z)}\rho(t,y)\,dy \, dz&=\frac{1}{|B_\delta|}\int_{B_r(x)}\int_{B_{r+\delta}(x)}\mathds{1}_{B_\delta(z)}(y)\rho(t,y)\,dy \, dz \\
		&=\frac{1}{|B_\delta|}\int_{B_{r+\delta}(x)}\rho(t,y)\int_{B_r(x)}\mathds{1}_{B_\delta(y)}(z)\, dz\,dy  \\
		&=\frac{1}{|B_\delta|}\int_{B_{r-\delta}(x)}\rho(t,y)\int_{B_r(x)}\mathds{1}_{B_\delta(y)}(z)\, dz\,dy  \\
		&+\frac{1}{|B_\delta|}\int_{B_{r+\delta}(x)\setminus B_{r-\delta}(x)}\rho(t,y)\int_{B_r(x)}\mathds{1}_{B_\delta(y)}(z)\, dz\,dy  \\
		&=\int_{B_{r-\delta}(x)}\rho(t,y)\,dy \\
		&+ \frac{1}{|B_\delta|}\int_{B_{r+\delta}(x)\setminus B_{r-\delta}(x)}\rho(t,y)\int_{B_r(x)}\mathds{1}_{B_\delta(y)}(z)\, dz\,dy \, ,
		\end{split}
		\end{equation}
                 where as usual $\mathds{1}_A(x) = 1$ if $x \in A$ and $\mathds{1}_A(x) = 0$ if $x \notin A$. 
		Namely, we proved the identity
		\begin{align}\label{migi}
		\int_{B_{r-\delta}(x)}  \rho(t,y)\,dy 
		& =
		\int_{B_r(x)}\fint_{B_{\delta}(z)}\rho(t,y)\,dy \, dz
		\\ \nonumber
               &	
               	- \frac{1}{|B_\delta|}\int_{B_{r+\delta}(x)\setminus B_{r-\delta}(x)}\rho(t,y)\int_{B_r(x)}\mathds{1}_{B_\delta(y)}(z)\, dz\,dy
		\end{align}
		Since
		\begin{equation}
		\begin{split}
		\left|\fint_{B_r(x)}\rho(t,y)\, dy\right|&\leq\frac{1}{|B_r|}\left|\int_{B_{r-\delta}(x)}\rho(t,y)\,dy\right|
		+\|\rho\|_{\infty}\frac{\left|B_r\setminus B_{r-\delta}\right|}{|B_r|}
		\end{split}
		\end{equation}
		using \eqref{migi}, triangle inequality and \eqref{UsingAlsoThisOne} we arrive to
		\begin{equation}
		\begin{split}
		\left|\fint_{B_r(x)}\rho(t,y)\, dy\right|
		&\leq \left|\fint_{B_r(x)}\fint_{B_{\delta}(z)}\rho(t,y)\,dy \, dz\right|+ \|\rho\|_{\infty}\frac{\left|B_{r+\delta}\setminus B_{r-\delta}\right|}{|B_r|}\\
		&+ \|\rho\|_{\infty}\frac{\left|B_r\setminus B_{r-\delta}\right|}{|B_r|}\\
		&\leq 
		\left(\frac{\kappa}{2}+4\frac{\delta}{r}+2\frac{\delta}{r}\right)\|\rho\|_{\infty} \, .
		\end{split}
		\end{equation}
                 Hence, using that $\delta=\delta(t)\leq \frac{\kappa}{12}r$, we conclude that
			\begin{equation}
			\left|\fint_{B_r(x)}\rho(t,y)\, dy\right|\leq\kappa \|\rho\|_{\infty}
			\end{equation}
			for all $x$ and all $t\geq \bar{t}$.
		
	\end{proof}


\begin{proposition}\label{NecCond}
Let $\rho_{0} \in L^{\infty}$ supported in $\overline{B_{1}}$ be a mean-free initial datum for which $\lim_{t \to \infty}\| \rho(t, \cdot) \|_{\dot{H}^{-1}} = 0$ or
$\lim_{t \to \infty} \mathcal{G}(\rho(t, \cdot)) = 0$ for all $\kappa \in (0,1)$. Then $\rho_0$ has to satisfy Assumption~\ref{equalradius3}.
\end{proposition}

\begin{proof}
Clearly it is sufficient to consider $r > 0$. Looking at the vector field $u$, it is immediate to check that the average of any solution $\rho(t, \cdot)$, advected by $u$, on any disk centered 
at the origin, is preserved. Namely we have, for all $r > 0$ (notice that $\rho(t,\cdot) = 0$ outside $B_{1}$): 
$$
\int_{|x| \leq r} \rho(t,x) dx = \int_{|x| \leq r} \rho_{0}(x) dx, \qquad \forall t \geq 0 \, .
$$ 
Thus, for any $r_{1} < r_{2}$, we still have 
\begin{equation}\label{GAsymptotic}
\int_{ r_{1} < |x| \leq r_{2}} \rho(t,x) dx = \int_{r_{1} < |x| \leq r_{2}} \rho_{0}(x) dx, \qquad \forall t \geq 0 \, .
\end{equation}
Using Proposition \ref{WeakDecay} in the case of the geometric mixing scale, we see that any of the assumptions of the current proposition imply that $\rho(t, \cdot)$ converges to zero weakly in $L^{2}$ as $t \to \infty$. Thus, testing 
against $\phi = \chi_{B_{r_{2}}}$ and $\phi=\chi_{B_{r_{1}}}$, we see that
\begin{equation}\nonumber
\lim_{t \to \infty}  \int_{r_{1} < |x| \leq r_{2}}  \rho(t, x) dx 
 = 
\lim_{t \to \infty} 
\left( \int_{|x| \leq r_{2}} \rho(t, x) dx - \int_{|x| \leq r_{1}} \rho(t, x) dx \right) = 0 \, ,
\end{equation} 
for all $r_{1} < r_{2}$. By \eqref{GAsymptotic}, this clearly implies  
\begin{equation}\label{PlusLebesgue}
\int_{r_{1} < r \leq r_{2}} \int_{\partial B_{r}}\rho_{0} \, dS_{r}  = \int_{r_{1} < |x| \leq r_{2}} \rho_{0}(x) dx = 0 \, ,
\end{equation}
for all $r_{1} < r_{2}$. Taking $r \in [r_{1}, r_{2}]$ and letting $r_{2} - r_{1} \to 0$, using \eqref{PlusLebesgue} and the Lebesgue differentiation theorem, 
we have proved that
\begin{equation}\label{QuasiFinito}
\int_{\partial B_r}\rho_0 \, dS_{r} = 0 \, ,
\end{equation}
for any $r > 0$ which is a Lebesgue point of the function $r > 0 \mapsto \int_{\partial B_r}\rho_0 \,dS_{r}$. Since this function is bounded
(recall that $\rho_{0} \in L^{\infty}$) and supported
on $(0,1]$, we obtain that \eqref{QuasiFinito} is valid for almost any $r > 0$, as claimed.
\end{proof}


\begin{thebibliography}{56}

	\bibitem{AlbCrippaCras}
	G. Alberti, G. Crippa and A. Mazzucato. Exponential self-similar mixing 
	and loss of regularity for continuity equations.
	\textit{C.R. Math. Acad. Sci. Paris, }
	252 (2014) no. 11, 901-906.
	
	\bibitem{AlbCrippa}
	G. Alberti, G. Crippa and A. Mazzucato. Exponential self-similar mixing by 
	incompressible flows.
	Journal of the American Math. Society, in press. Preprint. arXiv:1605.02090.
	
	 
	
	\bibitem{AlbCrippa2}
	G. Alberti, G. Crippa and A. Mazzucato. Loss of regularity for 
	continuity equations with non-Lipschitz velocity, 2018. arXiv:1802.0208. 
	
	\bibitem{Ambro}
	L. Ambrosio. Transport equation and Cauchy problem for BV vector fields.
	\textit{Invent. Math., }158(2):227-260, 2004.
	
	\bibitem{Bedrossian}
	 J. Bedrossian and M. Coti Zelati. Enhanced dissipation, 
	 hypoellipticity, and anomalous small noise inviscid limits in shear flows. 
	 \textit{Arch. Ration. Mech. Anal., } 224 (2017) no. 3, 1161-1204. 
	
	\bibitem{Bressan}
	A. Bressan.
	A lemma and a conjecture on the cost of rearrangements.
	\textit{Ren. Sem. Mat. Univ. Padova,}
	110:97-102, 2003.
	
	\bibitem{LellisCrippa}
	G. Crippa and C. De Lellis.
	Estimates and regularity results for the DiPerna Lions flow.
	\textit{J. Reine Angew. Math.,} 616:15-46, 2008.
	
	\bibitem{CellP}
	G. Crippa and C. Schulze.
	Cellular mixing with bounded palenstrophy.
	\textit{Math. Models and Methods in Appl. Sciences,} 27(12):2297-2320, 2017. 
	
		
	\bibitem{Depauw}
	N. Depauw.
	Non unicit\'e des solutions born\'ees pour un champ de 
	vecteurs BV en dehors d'un hyperplan.
	\textit{C.R. Math. Acad. Sci. Paris,} 337(4):249-252, 2003.
	
	\bibitem{Diperna}
	R. J. DiPerna and P.-L. Lions.
	Ordinary differential equations, transport theory and Sobolev spaces.
	\textit{Invent. Math., }98(3):511-547, 1989.
	
	\bibitem{EZ}
	T. M. Elgindi and A. Zlato\v{s}, Universal Mixers in All Dimensions, 2018. Preprint. arXiv:1809.09614. 
	
	
	\bibitem{Jabin}
	P.-E. Jabin. Critical non Sobolev regularity for continuity equations 
	with rough velocity fields.
	\textit{J. Diff. Equ., } 260:4739-4757, 2016.
	
	\bibitem{Kiselev}
	G. Iyer, A. Kiselev and X. Xu. Lower bounds on the mix norm of 
	passive scalars advected by incompressible enstrophy-constrained flows.
	\textit{Nonlinearity,} 27(5):973-985, 2014.
	
	\bibitem{leger}
	F. L\'eger.
	A new approach to bounds on mixing.
	\textit{Math. Models Methods Appl. Sci.,} 28(5):829-849, 2018.
	
	
	\bibitem{LinThif}
	Z. Lin, J.-L. Thiffeault and C.R. Doering.
	Optimal stirring strategies for passive scalar mixing.
	\textit{J. Fluid Mech., }675:465-476, 2011.
	
	\bibitem{Zeng}
	Z. Lin and C. Zeng. Inviscid Dynamical Structures Near Couette Flow.
	\textit{Arch. Ration. Mech. Anal., } 200 (2011) no. 3, 1075-1097. 
	
	\bibitem{Lin}
	E. Lunasin, Z. Lin, A. Novikov, A. Mazzucato, and C.R. Doering.
	Optimal mixing and optimal stirring for fixed energy, fixed power, or 
	fixed palenstrophy flows.
	\textit{J. Math. Phys., }53(11):115611, 2012.
	
	\bibitem{Multi}
	G. Mathew, I. Mezi\'c and L. Petzold.
	A multiscale measure for mixing.
	\textit{Phys. D,} 211(1-2):23-46, 2005.

	\bibitem{Seis}
	C. Seis.
	Maximal mixing by incompressible fluid flows.
	\textit{Nonlinearity,}
	26(12):3279-3289, 2013.
	
	\bibitem{Stein}
	E. Stein.
	Singular Integrals and Differentiability Properties of Functions.
	\textit{Princeton Univ. Press,} 1970.
	
	\bibitem{Thieff}
	J.-L. Thieffault.
	Using multiscale norms to quantify mixing and transport.
	\textit{Nonlinearity,}
	84(3), R1-R44, 2012.
	
	\bibitem{Yao}
	Y. Yao and A. Zlato\v{s}.
	Mixing and Un-mixing by Incompressbile Flows. 
	\textit{J. Eur. Math. Soc. (JEMS),} 19(7):1911--1948, 2017.

	
	\bibitem{Zill1}
	C. Zillinger. Linear inviscid damping for monotone shear flows.
         \textit{Trans. Amer. Math. Soc.,} 369(12):8799--8855, 2017.

	\bibitem{Zill2}
	C. Zillinger. On circular flows: linear stability and damping.
	\textit{J. Differential Equations,} 263(11):7856-7899, 2017. 
		

         \bibitem{Zill4} C. Zillinger. On geometric and analytic mixing scales: comparability 
         and convergence rates for transport problems, 2018. Preprint. arXiv:1804.11299.

\end{thebibliography}
\end{document}